\documentclass[a4paper,reqno]{amsart}

\usepackage{units}
\usepackage{setspace}
\usepackage[totalwidth=14cm,totalheight=22cm]{geometry}
\usepackage[T1]{fontenc}
\usepackage[dvips]{graphicx}
\usepackage{epsfig}
\usepackage{amsmath,amssymb,amscd,amsthm}
\usepackage{subfigure}
\usepackage[latin1]{inputenc}
\usepackage{latexsym}
\usepackage{graphics}
\usepackage{colortbl}
\usepackage{color}
\usepackage{ae}
\usepackage{enumitem}
\usepackage[all]{xy}

\makeatletter
\newtheorem*{rep@theorem}{\rep@title}
\newcommand{\newreptheorem}[2]{%
\newenvironment{rep#1}[1]{%
 \def\rep@title{#2 \ref{##1}}%
 \begin{rep@theorem}}%
 {\end{rep@theorem}}}
\makeatother

\makeatletter
\newtheorem*{rep@cor}{\rep@title}
\newcommand{\newrepcor}[2]{%
\newenvironment{rep#1}[1]{%
 \def\rep@title{#2 \ref{##1}}%
 \begin{rep@cor}}%
 {\end{rep@cor}}}
\makeatother

\makeatletter
\newtheorem*{rep@prop}{\rep@title}
\newcommand{\newrepprop}[2]{%
\newenvironment{rep#1}[1]{%
 \def\rep@title{#2 \ref{##1}}%
 \begin{rep@prop}}%
 {\end{rep@prop}}}
\makeatother

\newtheorem{cor}{Corollary}[section]
\newtheorem{corx}[cor]{Corollary}

\newtheorem{theorem}[cor]{Theorem}
\newtheorem{thmx}[cor]{Theorem}

\newtheorem{prop}[cor]{Proposition}

\newrepcor{cor}{Corollary}
\newreptheorem{theorem}{Theorem}
\newrepprop{prop}{Proposition}

\newtheorem{lemma}[cor]{Lemma}

\theoremstyle{definition}
\newtheorem{defi}[cor]{Definition}
\theoremstyle{remark}
\newtheorem{remark}[cor]{Remark}
\newtheorem*{remark*}{Remark}
\newtheorem{example}[cor]{Example}

\newcommand{\cD}{{\mathcal D}}

\newcommand{\cF}{{\mathcal F}}

\newcommand{\C}{{\mathbb C}}

\newcommand{\R}{{\mathbb R}}

\newcommand{\Z}{{\mathbb Z}}

\newcommand{\Hom}{\mathrm{Hom}}

\newcommand{\dR}{\mathrm{dR}}

\newcommand{\Hyp}{\mathbb{H}}
\newcommand{\AdS}{\mathbb{A}\mathrm{d}\mathbb{S}}

\newcommand{\psl}{\mathfrak{sl}}

\newcommand{\PSL}{\mathrm{PSL}}
\newcommand{\PSLR}{\mathrm{PSL}_2\R}
\newcommand{\USL}{\widetilde{\mathrm{SL}_2\R}}
\newcommand{\ML}{{\mathrm{M}\!\mathrm{L}}}

\newcommand{\TA}{T^{1}_\star\tAdS}
\newcommand{\hol}{\mathrm{hol}}

\newcommand{\isom}{\mathrm{Isom}}

\newcommand{\tAdS}{\widetilde{\AdS^3}}

\newcommand{\RP}{\R \mathrm{P}}
\newcommand{\Flux}{\mathrm{Flux}}
\newcommand{\CFlux}{\widehat{\mathrm{Flux}}}
\newcommand{\Symp}{\mathrm{Symp}}
\newcommand{\Ham}{\mathrm{Ham}}

\begin{document}

\setcounter{secnumdepth}{3}
\setcounter{tocdepth}{2}

\title[Equivariant maps into $\AdS^3$ and the symplectic geometry of $\mathbb H^2\times \mathbb H^2$]{Equivariant maps into Anti-de Sitter space and the symplectic geometry of $\mathbb H^2\times \mathbb H^2$}

\author[Francesco Bonsante]{Francesco Bonsante}
\address{Francesco Bonsante: Dipartimento di Matematica ``Felice Casorati", Universit\`{a} degli Studi di Pavia, Via Ferrata 5, 27100, Pavia, Italy.} \email{bonfra07@unipv.it} 
\author[Andrea Seppi]{Andrea Seppi}
\address{Andrea Seppi: Dipartimento di Matematica ``Felice Casorati", Universit\`{a} degli Studi di Pavia, Via Ferrata 5, 27100, Pavia, Italy.} \email{andrea.seppi01@ateneopv.it}

\thanks{The authors were partially supported by FIRB 2010 project ``Low dimensional geometry and topology'' (RBFR10GHHH003) and by
PRIN 2012 project ``Moduli strutture algebriche e loro applicazioni''.
The authors are members of the national research group GNSAGA}

\begin{abstract}
Given two Fuchsian representations $\rho_l$ and $\rho_r$ of the fundamental group of a closed oriented surface $S$ of genus $\geq 2$, we study the relation between Lagrangian submanifolds of $M_\rho=(\mathbb{H}^2/\rho_l(\pi_1(S)))\times (\mathbb{H}^2/\rho_r(\pi_1(S)))$ and $\rho$-equivariant embeddings $\sigma$ of $\widetilde S$ into Anti-de Sitter space, where $\rho=(\rho_l,\rho_r)$ is the corresponding representation into $\mathrm{PSL}_2\mathbb R\times \mathrm{PSL}_2\mathbb R$. It is known that, if $\sigma$ is a maximal embedding, then its Gauss map takes values in the unique minimal Lagrangian submanifold $\Lambda_{\mathrm{ML}}$ of $M_\rho$.

We show that, given any $\rho$-equivariant embedding $\sigma$, its Gauss map gives a Lagrangian submanifold Hamiltonian isotopic to $\Lambda_{\mathrm{ML}}$. Conversely, any Lagrangian submanifold Hamiltonian isotopic to $\Lambda_{\mathrm{ML}}$ is associated to some equivariant embedding into the future unit tangent bundle of the universal cover of Anti-de Sitter space.
\end{abstract}

\maketitle

\section{Introduction}

Anti-de Sitter space of dimension three, denoted $\AdS^3$, is the Lie group $\PSL_2\R$ endowed with (a multiple of) the Killing form. It is a Lorentzian manifold of constant negative sectional curvature, and the identity component of its isometry group is naturally isomorphic to $\PSL_2\R\times \PSL_2\R$. Its study, from the purely mathematical viewpoint, was initiated by the work of Mess \cite{Mess}, and has remarkably spread since then. One of the reasons for this interest, which is the main motivation behind this paper, is the relation between equivariant embedded surfaces in Anti-de Sitter space  and hyperbolic geometry in dimension 2. 

Let $S$ be a closed oriented surface of negative Euler characteristic.
We will consider $\rho$-equivariant spacelike (i.e. having Riemannian induced metric) embeddings of the universal cover $\widetilde S$ of $S$ into $\AdS^3$, where $\rho=(\rho_l,\rho_r):\pi_1(S)\to \PSL_2\R\times \PSL_2\R$ is a representation of the fundamental group of $S$. Mess proved that in this case, $\rho_l$ and $\rho_r$ are Fuchsian representations.
 Hence (identifying $\PSL_2\R$ with the isometry group of the hyperbolic plane) $\Hyp^2/\rho_l(\pi_1(S))$ and $\Hyp^2/\rho_r(\pi_1(S))$
are closed hyperbolic surfaces, which we identify with $(S,h_l)$ and $(S,h_r)$, for $h_l$ and $h_r$ hyperbolic metrics on $S$. 

Mess considered the case of \emph{pleated} immersions $\sigma:\widetilde S\to\AdS^3$. In this case, there is a well-defined hyperbolic metric $h_0$ on $S$ induced by $\sigma$, and by a rather explicit geometric construction, Mess defined two maps $e_\lambda^l:(S,h_0)\to (S,h_l)$ and $e_\lambda^r:(S,h_0)\to (S,h_r)$ which turn out to be left and right \emph{earthquake maps}, where the the bending lamination $\lambda$ of $\sigma$ coincides with the earthquake lamination. This construction gave a new interpretation to Thurston's proof of the existence of earthquake maps between two fixed closed hyperbolic surfaces homeomorphic to $S$. 

Later, Krasnov and Schlenker in \cite{Schlenker-Krasnov} studied $\rho$-equivariant embeddings $\sigma:\widetilde S\to\AdS^3$ under the condition that the induced metric has negative curvature. In this case, one can define two diffeomorphisms $\varphi^l:S\to (S,h_l)$ and $\varphi^r:S\to (S,h_r)$, which are somehow the smooth analogues of the earthquake maps above. The composition $\varphi:=\varphi^r\circ (\varphi^l)^{-1}$ turns out to be an area-preserving diffeomorphism between $(S,h_l)$ and $(S,h_r)$. 
In other words, the graph of $\varphi$ is a Lagrangian submanifold of $(S\times S,\Omega_\rho)$, where $\Omega_\rho$ is the natural symplectic form:
$$\Omega_\rho=p_l^*\Omega_l-p_r^*\Omega_r~,$$
if $\Omega_l$ and $\Omega_r$ are the area forms of the hyperbolic metric $h_l$ and $h_r$.

A remarkable application of this construction comes from the particular case of \emph{maximal}
equivariant embeddings (i.e. of vanishing mean curvature). This case is characterized by the fact that the associated area-preserving diffeomorphism $\varphi$ is \emph{minimal Lagrangian} (\cite{Schlenker-Krasnov}, see also \cite{bon_schl} and \cite{torralbo}). This means that the Lagrangian submanifold ${graph}(\varphi)$ is minimal in $(S\times S,h_l\oplus h_r)$, see also \cite{labourieCP} and \cite{Schoenharmonic}. 

Other applications were given in \cite{bbzads,bonseppitamb,Schlenker-Krasnov} for the case of closed surfaces, \cite{bon_schl,ksurfaces,seppimaximal,seppiminimal,scarinci} for generalizations to maps from the hyperbolic plane $\Hyp^2$ to itself in the context of universal Teichm\"uller space, and \cite{bms,bms2,ksurfaces} for a more general class of maps called \emph{landslides}.

The above characterization of minimal Lagrangian maps highlights the fact, which is behind the main ideas of this paper, that the most relevant object to associate to $\sigma$ is a Lagrangian submanifold of $(S\times S,\Omega_\rho)$, rather than the diffeomorphism $\varphi$. 
The construction of the diffeomorphism $\varphi$ can be generalized to any $\rho$-equivariant spacelike embedding $\sigma:\widetilde S\to\AdS^3$ (without any assumption on the curvature), thus providing a smooth Lagrangian submanifold of the symplectic manifold $(S\times S,\Omega_\rho)$, which we will denote $\Lambda_\sigma$. This has been observed in \cite{barbotkleinian}. In this paper, we provide an independent proof in a framework which is better adapted for our aims.

The problem we address in this paper is to what extent this construction can be reversed, that is: is it possible to associate to a Lagrangian submanifold a $\rho$-equivariant spacelike embedding? In \cite[§3]{barbotkleinian} it was showed that the unique \emph{local} obstruction to inverting the construction is given by the condition of being a Lagrangian submanifold. Hence it remains to understand the \emph{global} problem, and we will give an obstruction purely in terms of the symplectic geometry of $(S\times S,\Omega_\rho)$. This will be the content of Corollary \ref{main cor}. Corollary \ref{main cor} is actually a consequence of our main result (Theorem \ref{main thm}), which also shows that this obstruction is complete in a slightly more general context.

\subsection*{Statement of the main result}

The aforementioned obstruction will be given in terms of \emph{Hamiltonian submanifolds} and \emph{Hamiltonian symplectomorphisms}. Let us briefly recall these notions.
If we fix a representation $\rho=(\rho_l,\rho_r):\pi_1(S)\to\PSL_2\R\times \PSL_2\R$ with $\rho_l$ and $\rho_r$ Fuchsian, let $\Phi_t:(S\times S,\Omega_\rho)\to (S\times S,\Omega_\rho)$ be a smooth path of symplectomorphisms with $\Phi_0=\mathrm{id}$. The path $\Phi_t$ is called \emph{Hamiltonian} if there exists a smooth family of functions $H_t:S\times S\to\R$ such that the generating vector field of $\Phi_t$ is the symplectic gradient of $H_t$. Then a symplectomorphism $\Phi:(S\times S,\Omega_\rho)\to (S\times S,\Omega_\rho)$ is \emph{Hamiltonian} if there exists a Hamiltonian path $\Phi_t$ such that $\Phi_0=\mathrm{id}$ and $\Phi_1=\Phi$.
It turns out that Hamiltonian symplectomorphisms form a normal subgroup $\Ham(S\times S,\Omega_\rho)$ of $\Symp_0(S\times S,\Omega_\rho)$ (the identity component of $\Symp(S\times S,\Omega_\rho)$). See also the discussion below on the flux homomorphism.

The group $\Symp_0(S\times S,\Omega_\rho)$ clearly acts on the space of smooth Lagrangian submanifolds of $(S\times S,\Omega_\rho)$ isotopic to the diagonal. A consequence of the main result in this paper --- namely Theorem \ref{main thm} --- is that the submanifolds $\Lambda_\sigma$ associated to smooth $\rho$-equivariant spacelike embeddings $\sigma:\widetilde S\to\AdS^3$ lie in a single orbit of the action of $\Ham(S\times S,\Omega_\rho)$. Since we already know that the unique minimal Lagrangian submanifold $\Lambda_\ML$ of $(S\times S,\Omega_\rho)$ isotopic to the diagonal is the submanifold associated to the maximal $\rho$-equivariant embedding, we obtain:

\begin{corx} \label{main cor}
Let $\rho=(\rho_l,\rho_r):\pi_1(S)\to\PSL_2\R\times \PSL_2\R$, where $\rho_l$ and $\rho_r$ are Fuchsian representations. Then for every $\rho$-equivariant spacelike embedding $\sigma:\widetilde S\to\AdS^3$, $\Lambda_\sigma$ is Hamiltonian isotopic to the unique minimal Lagrangian submanifold $\Lambda_\ML$ isotopic to the diagonal.
\end{corx}

Two Lagrangian submanifolds $\Lambda_0,\Lambda_1$ are called \emph{Hamiltonian isotopic} if there exists a Hamiltonian symplectomorphism $\Phi\in\Ham(S\times S,\Omega_\rho)$ such that $\Phi(\Lambda_0)=\Lambda_1$. 

Before actually stating the main theorem of the paper (Theorem \ref{main thm}), of which Corollary \ref{main cor} will be a consequence, let us explain a geometric interpretation of the submanifold $\Lambda_\sigma$ associated to $\sigma:\widetilde S\to\AdS^3$.

The basic observation is that there is a natural projection $\pi:\TA\to \Hyp^2\times \Hyp^2$, where $\TA$ is the bundle of future-directed unit timelike vectors in the tangent bundle of the universal cover $\tAdS$ of $\AdS^3$. This projection $\pi$ has the important property that its fibers are the orbits of the geodesic flow on $\TA$. Now, given a $\rho$-equivariant spacelike embedding $\sigma:\widetilde S\to\AdS^3$, we can consider the normal section $\widetilde\sigma_N:\widetilde S\to\TA$ of any lift $\widetilde\sigma:\widetilde S\to\tAdS$. It turns out that the composition $\pi\circ\widetilde\sigma_N$ does not depend on the chosen lift $\widetilde\sigma$ and provides a $\rho$-equivariant Lagrangian embedding of $\widetilde S$ into $\Hyp^2\times \Hyp^2$. Thus  $\pi\circ\widetilde\sigma_N$ induces an embedding of $S$ into $(S\times S,\Omega_\rho)$, whose image is the Lagrangian submanifold $\Lambda_\sigma$ we have already mentioned. If the curvature of the metric induced by the embedding $\sigma$ is negative, then $\Lambda_\sigma$ is transverse to the product structure of $S\times S$ and thus it is globally a graph, hence recovering the construction of Krasnov and Schlenker.

It turns out that the normal section $\widetilde\sigma_N$ is orthogonal to the orbits of the geodesic flow for the Sasaki metric of $\TA$ and is equivariant for a uniquely determined lift $\widetilde\rho:\pi_1(S)\to\isom_0(\tAdS)$, which we will call \emph{standard lift}. Such a lift $\widetilde \rho$ only depends on $\rho$, and not on $\sigma$ or on the choice of its lift. Conversely, a $\widetilde\rho$-equivariant embedding $\widetilde\sigma:\widetilde S\to \TA$ is the normal section of some embedding $\sigma:\widetilde S\to\AdS^3$ if and only if:
\begin{enumerate}
\item $\widetilde\sigma$ is orthogonal to the orbits of the geodesic flow for the Sasaki metric on $\TA$;
\item $\widetilde\sigma$ is transverse to the fibers of the projection $\TA\to\AdS^3$.
\end{enumerate}
Our main theorem is a complete characterization of the Lagrangian submanifolds 
$$\Lambda_{\widetilde\sigma}=\textrm{image}(\pi\circ \widetilde\sigma)/\rho(\pi_1(S)) \subset(S\times S,\Omega_\rho)$$ induced by the $\widetilde\rho$-equivariant embedding $\widetilde \sigma$ satisfying condition $(1)$ above.

\begin{thmx} \label{main thm}
Let $\rho=(\rho_l,\rho_r):\pi_1(S)\to\PSL_2\R\times \PSL_2\R$, where $\rho_l$ and $\rho_r$ are Fuchsian representations, and let $\widetilde\rho:\pi_1(S)\to\isom_0(\tAdS)$ be its standard lift. Then
$$\left\{\Lambda_{\widetilde\sigma}:\begin{aligned} \widetilde\sigma &\text{ is a }\widetilde\rho\text{-equivariant embedding orthogonal} \\
&\text{to the orbits of the geodesic flow} \end{aligned} \right\}=\Ham(S\times S,\Omega_\rho)\cdot\Lambda_\ML~,$$
where $\Lambda_\ML$ is the unique minimal Lagrangian submanifold of $(S\times S,\Omega_\rho)$ isotopic to the diagonal.
\end{thmx}

Corollary \ref{main cor} follows from one of the two inclusions of Theorem \ref{main thm}, by taking the normal section $\widetilde \sigma_N$ of a $\rho$-equivariant embedding $\sigma$, and observing that by construction $\Lambda_{\widetilde\sigma}=\Lambda_\sigma$.

\subsection*{Organization and techniques involved}
Let us now highlight the main tools which are used in the proof of Theorem \ref{main thm}, and how they are organized in the paper.
\\

In Section \ref{sec preliminaries}, we introduce Anti-de Sitter space, its universal cover, its future unit tangent bundle, and their geometric properties. The proof of Theorem \ref{main thm} will then be based on a translation of the problem into the language of \emph{principal} $\R$\emph{-bundles} over $(S\times S,\Omega_\rho)$, and in studying the \emph{symplectic geometry} of the base $(S\times S,\Omega_\rho)$.
\\

Hence the main purpose of Section \ref{sec principal bundles} is the construction of a principal $\R$-bundle $P_\rho$ over $(S\times S,\Omega_\rho)$, and the study of its geometry. This is simply obtained as the quotient of the principal $\R$-bundle 
$$\pi:\TA\to\Hyp^2\times \Hyp^2~,$$ 
where the $\R$-action on $\TA$ is given by geodesic flow. The actions of $\widetilde \rho(\pi_1(S))$ on $\TA$ and of $\rho(\pi_1(S))$ on $\Hyp^2\times \Hyp^2$ are free and properly discontinous, and $\pi$ is equivariant with respect to these actions. Hence, taking the quotient, we have a principal $\R$-bundle
$$\pi_\rho:P_\rho \to (S, h_l)\times (S,h_r)~.$$
The key observation here is that there is a natural principal $\R$-connection $\omega_\rho$ on $P_\rho$, for which parallel (local) sections are precisely sections orthogonal to the orbits of the geodesic flow. Hence $\widetilde\rho$-equivariant embeddings $\widetilde \sigma:\widetilde S\to\TA$ correspond precisely to parallel  sections of $\pi_\rho$ over $\Lambda_{\widetilde\sigma}\subset (S\times S,\Omega_\rho)$.

The main result of Section \ref{sec principal bundles} is then Proposition \ref{prop curvature}, which shows that the curvature of $(P_\rho,\omega_\rho)$ --- which is a $\R$-valued 2-form $R_\rho\in\Omega^2(P_\rho,\R)$ --- is:
$$R_\rho=\frac{1}{2}\pi_\rho^*\Omega_\rho~.$$
Hereafter, let us denote by $\Lambda:S\to (S\times S,\Omega_\rho)$ an embedding, rather than the embedded surface as above. Hence we deduce that $\Lambda$ is a Lagrangian embedding if and only if the bundle $\Lambda^*P_\rho$ is a flat $\R$-bundle, for the connection induced by $\omega_\rho$. In particular we recover the fact that, when the embedding $\Lambda_\sigma$ associated to $\sigma:\widetilde S\to\AdS^3$ is the graph of some diffeomorphism $\varphi:(S,h_l)\to (S,h_r)$ (that happens when the metric induced by $\sigma$ has negative curvature), then $\varphi$ is a symplectomorphism (i.e. it is area-preserving).
\\

As a consequence, we are interested in a characterization of those Lagrangian embeddings $\Lambda$ such that the flat $\R$-bundle $\Lambda^*P_\rho$ admits a global parallel section. Equivalently, we need to characterize the condition that the holonomy of $\Lambda^*P_\rho$ is trivial.
In Section \ref{sec holonomy} we solve this step by studying the symplectic geometry of $(S\times S,\Omega_\rho)$. 

The main tool we use here is an adaptation (given in \cite{MR3124936}) to the context of Lagrangian submanifolds of the \emph{flux homomorphism} introduced by Calabi in \cite{MR0350776}, see also \cite{MR490874} and \cite[Chapter 6]{MR1698616}. The latter is a surjective group homomorphism 
$$\CFlux:\widetilde \Symp_0(S\times S,\Omega_\rho)\to H^1_\dR(S\times S,\R)~,$$
whose kernel is $\widetilde \Ham(S\times S,\Omega_\rho)$. Then in \cite{MR3124936}, a generalization was given, as a map $\Flux$ which associates to a smooth path of Lagrangian embeddings $\Lambda_t$ an element in $\Hom(\pi_1(S),\R)$. In our case, this only depends on the endpoints $\Lambda_0$ and $\Lambda_1$, and if $\Lambda_1=\Phi_1(\Lambda_0)$ for some path $\Phi_t\in\Symp_0(S\times S,\Omega_\rho)$ with $\Phi_0=\mathrm{id}$, then $\Flux([\Lambda_\bullet])$ coincides with $\Lambda_0^*\CFlux([\Phi_t])$ under the usual identification $H^1_\dR(S,\R)\cong\Hom(\pi_1(S),\R)$.

The most important step of Section \ref{sec holonomy} is Proposition \ref{prop holonomy flux}. It shows that, if $\hol_F$ denotes the holonomy representation of a flat $\R$-bundle $F$ over $S$, then 
$$\hol_{\Lambda_1^*P_\rho}=\hol_{\Lambda_0^*P_\rho}+\Flux([\Lambda_\bullet])~.$$
Hence the flat subbundles $\Lambda^*P_\rho$ which have the same holonomy form precisely an orbit of the action of $\Ham(S\times S,\Omega_\rho)$ on the space of Lagrangian embeddings of $S$ into $(S\times S,\Omega_\rho)$. 
\\

Section \ref{sec proof} then puts together the above ingredients to prove Theorem \ref{main thm}. In particular, since we already know that the minimal Lagrangian embedding $\Lambda_\ML$ is associated to the $\rho$-equivariant maximal embedding, $\Lambda_\ML^* P_\rho$ is a trivial flat $\R$-bundle (i.e. it admits a parallel global section). Hence by the above results, for an embedding $\Lambda:S\to (S\times S,\Omega_\rho)$, $\Lambda^*P_\rho$ is a trivial flat bundle is and only if $\Lambda$ is Lagrangian \emph{and} the map $\Flux$ vanishes on the class of a path $\Lambda_t$ such that $\Lambda_0=\Lambda_\ML$ and $\Lambda_1=\Lambda$. This is equivalent to the fact that $\Lambda$ is in the $\Ham(S\times S,\Omega_\rho)$-orbit of $\Lambda_\ML$.

\subsection*{Acknowledgements}
We would like to thank Alessandro Ghigi for several explanations about symplectic geometry, in particular for pointing out the relevant notion of flux homomorphism. Moreover, we would like to thank Thierry Barbot for interesting discussions and for pointing out relevant references.



\section{Anti-de Sitter geometry} \label{sec preliminaries}


In this section we introduce the geometry of Anti-de Sitter space of dimension three, of its universal cover $\tAdS$, and of its future unit timelike tangent bundle $\TA$. We prove a simple --- though essential in the following --- lifting statement (Lemma \ref{lemma lifting}) which defines the \emph{standard lift} of a representation $\rho=(\rho_l,\rho_r):\pi_1(S)\to\PSL_2\R\times \PSL_2\R$, with $\rho_l$ and $\rho_r$ Fuchsian.

\subsection{Anti-de Sitter space}
Let us consider the Killing form $\kappa$ on the Lie group $\PSLR$, which is a bi-invariant bilinear form on the Lie algebra $\psl_2\R$ of signature $(2,1)$. The bilinear form $\kappa$ induces a Lorentzian metric on $\PSLR$, which we denote by $g_\kappa$. Then we define Anti-de Sitter space of dimension 3 as:
$$\AdS^3:=\left(\PSLR, \frac{1}{8}g_\kappa\right)~.$$
Hence $\AdS^3$ is a Lorentzian manifold, orientable and time-orientable, topologically a solid torus, of constant sectional curvature. Due to the normalization factor $1/8$, the sectional curvature is $-1$ (\cite[Lemma 2.1]{ksurfaces}). 
By construction, the identity component of the isometry group of $\AdS^3$ is:
$$\isom_0(\AdS^3)\cong \PSLR\times\PSLR~,$$
where a pair $(\alpha,\beta)\in\PSLR\times\PSLR$ acts on $\AdS^3$ by
\begin{equation} \label{eq isometry action}
(\alpha,\beta)\cdot \gamma=\alpha\circ\gamma\circ \beta^{-1}~.
\end{equation}

Recall that $\PSLR$ is identified to the group of orientation-preserving isometries of the hyperbolic plane, in the upper-half plane model:
$$\Hyp^2:=\left(\{z\in\C\,\mid\,\Im(z)>0\},\frac{|dz|^2}{\Im(z)^2}\right)~.$$
The identification of $\isom_0(\Hyp^2)$ with $\PSLR$ is defined by associating to an isometry of $\Hyp^2$ its extension to the visual boundary $\partial_\infty\Hyp^2=\RP^1$, which is a projective transformation.

A differentiable curve $\gamma:I\to\AdS^3$ is \emph{timelike} if $g_\kappa(\dot\gamma,\dot\gamma)<0$ at every point $\gamma(t)$, \emph{spacelike} if $g_\kappa(\dot\gamma,\dot\gamma)>0$, and \emph{lightlike} if $g_\kappa(\dot\gamma,\dot\gamma)=0$. We claim that every timelike geodesic in $\AdS^3$ has the form:
\begin{equation} \label{eq closed geodesic}
L_{x,y}:=\{\gamma\in\PSLR\,\mid\,\gamma(y)=x\}~,
\end{equation}
for $x,y\in\Hyp^2$.
Indeed, $L_{x_0,x_0}$ is a closed timelike geodesic for every $x_0\in\Hyp^2$, since it is a maximal compact subgroup, hence the induced bilinear form is negative definite. It is a geodesic since it is a 1-parameter group, and the Riemannian exponential map coincides with the Lie group exponential map as $g_\kappa$ is a bi-invariant metric. 
It also turns out that it has length $\pi$, the arclength parameter being $1/2$ the angle of rotation of elliptic elements fixing $x_0$.

It can be easily checked that 
\begin{equation} \label{eq transformation geodesics}
(\alpha,\beta)\cdot L_{x,y}=L_{\alpha(x),\beta(y)}~,
\end{equation}
and this also shows that every timelike geodesic of $\AdS^3$ is of the form $L_{x,y}$ for some $x,y$, since it is the image of $L_{x_0,x_0}$ under some isometry $(\alpha,\beta)$ of $\AdS^3$.
Hence there is a 1-1 correspondence
\begin{equation} \label{eq bij}
\Hyp^2\times\Hyp^2\leftrightarrow\{\text{timelike geodesics in }\AdS^3\}~,
\end{equation}
defined by 
$(x,y)\mapsto L_{x,y}$.
As a consequence of Equation \eqref{eq transformation geodesics}, the bijection of Equation \eqref{eq bij} is equivariant with respect to the action of $\PSLR\times\PSLR$ on $\Hyp^2\times\Hyp^2$ by isometries of $\Hyp^2$ on each factor, and on the set of timelike geodesics induced by isometries of $\AdS^3$.

Anti-de Sitter space is naturally endowed with a boundary, which is defined in the following way:
$$\partial_\infty\AdS^3:=\RP^1\times \RP^1~,$$
where a sequence $\gamma_n\in\PSLR$ converges to a pair $(p,q)\in\RP^1\times \RP^1$ if and only if there exists a point $x\in\Hyp^2$ such that:
\begin{equation} \label{eq convergence boundary}
\begin{cases}
\gamma_n(x)\to p \\
\gamma_n^{-1}(x)\to q
\end{cases}~.
\end{equation}
It is a classical fact that, if Equation \eqref{eq convergence boundary} holds for some $x\in\Hyp^2$, then it holds for every $x\in\Hyp^2$. Naturality of the boundary means that every isometry of $\AdS^3$, of the form $(\alpha,\beta)\in\PSLR\times\PSLR$, extends to the boundary by means of the obvious action of $(\alpha,\beta)$ on $\RP^1\times \RP^1$.

\subsection{Equivariant embeddings for closed surfaces} \label{subsec embeddings}

In this paper, $S$ will be a closed, oriented surface of negative Euler characteristic. We will be particularly interested in $\rho$-equivariant spacelike immersions $\sigma:\widetilde S\to\AdS^3$, where $\widetilde S$ is the universal cover of $S$, and $\rho:\pi_1(S)\to \PSLR\times\PSLR$ is a representation of the fundamental group of $S$. This means that $\sigma$ satisfies
$$\sigma\circ\tau=\rho(\tau)\circ\sigma~,$$
for every $\tau\in\pi_1(S)$, where the action of $\tau$ on $\widetilde S$ is by deck transformation.
 In \cite{Mess}, Mess proved that:
\begin{itemize}
\item If $\sigma$ is a $\rho$-equivariant immersion, then it is a proper embedding (Lemma 6).
\item In this case, $\rho=(\rho_l,\rho_r)$ where $\rho_l,\rho_r:\pi_1(S)\to\PSLR$ are Fuchsian representations (Proposition 19).
\item Moreover, $\sigma$ extends to the visual boundary $\partial_\infty \widetilde S$, and $\sigma(\partial_\infty \widetilde S)$ is the graph in $\partial_\infty \AdS^3=\RP^1\times \RP^1$ of the unique orientation-preserving homeomorphism which conjugates  $\rho_l$ to $\rho_r$ (Proposition 20).
\end{itemize}

Hence in this paper we will always consider representations $$\rho=(\rho_l,\rho_r):\pi_1(S)\to \isom_0(\AdS^3)\cong \PSLR\times\PSLR$$ with $\rho_l,\rho_r$ Fuchsian. 
As mentioned in the introduction, we will consider equivariant embeddings in the universal cover of $\AdS^3$, rather than in $\AdS^3$ itself. The motivation for this is the fact that, if $\sigma$ is a spacelike embedding in $\AdS^3$ equivariant for a representation $\rho=(\rho_l,\rho_r)$, then $\sigma$ lifts to an embedding $\widetilde\sigma$ in $\tAdS$, invariant by a \emph{standard lift} of $\rho$ to a representation of $\pi_1(S)$ in the isometry group of $\tAdS$. This will be explained in the next subsections.


\subsection{The universal cover of $\AdS^3$}

We will denote $\tAdS$ the metric universal cover of $\AdS^3$: if 
$$p:\USL\to\PSLR$$
is the universal covering map of $\PSLR$, then
$$\tAdS:=\left(\USL,\frac{1}{8}p^*g_\kappa\right)~.$$
Namely, we endow $\USL$ with the Lorentzian metric induced by the Killing form on the Lie algebra, which is also the metric which makes $p$ a local isometry. Let us observe that $\USL$ is a central extension of $\PSLR$. More precisely, if we denote $\mathcal Z:=\mathrm{Z}(\USL)$ the center of the Lie group $\USL$, the following is a short exact sequence of groups:
\[
\xymatrix{
1\ar[r] &\mathcal Z\ar[r] &\USL\ar[r]^p & \PSLR\ar[r]  & 1
}~.
\]
The fundamental group $\pi_1(\AdS^3)\cong\Z$ is thus identified to $\mathcal Z$, and left multiplication by an element $\xi$ of $\mathcal Z$ (which is the same as right multiplication, since $\xi$ is in the center) corresponds to the action on $\USL$ by deck transformations.

Hence it turns out that the identity component of the isometry group of $\tAdS$ is:
\begin{equation} \label{eq isometry universal ads}
\isom_0(\tAdS)=(\USL\times\USL)/\Delta_{\mathcal Z}~,
\end{equation}
where $\Delta_{\mathcal Z}$ denotes the diagonal of $\mathcal Z$:
$$\Delta_{\mathcal Z}:=\{(\xi,\xi)\,\mid\,\xi\in\mathcal Z\}<\USL\times\USL~.$$

Observe that a generator of the fundamental group $\pi_1(\PSLR)\cong\Z$ is represented by loops whose image is a closed timelike geodesic of the form $L_{x,y}$ (recall Equation \eqref{eq closed geodesic}). Hence timelike geodesics of $\tAdS$ are the preimages of timelike geodesics of $\AdS^3$, and are copies of $\R$ (with the negative metric).

Therefore we have a 1-1 correspondence
\begin{equation} \label{eq bij univ}
\Hyp^2\times\Hyp^2\leftrightarrow\{\text{timelike geodesics in }\tAdS\}~,
\end{equation}
which is defined by associating to $(x,y)\in\Hyp^2\times\Hyp^2$ the timelike geodesic 
$p^{-1}(L_{x,y})$. 
This correspondence is again equivariant, with respect to the action of $(\USL\times\USL)/\Delta_{\mathcal Z}$ on the set of timelike geodesics induced by isometries of $\tAdS$, and the action on $\Hyp^2\times\Hyp^2$:
\begin{equation} \label{eq usl on hyp2}
(\alpha,\beta)\cdot (x,y)=(p(\alpha)(x),p(\beta)(y))~.
\end{equation}
 In fact, \eqref{eq usl on hyp2} is well-defined since elements in the center $\mathcal Z=\mathrm{Ker}(p)$ act trivially on $\Hyp^2$.

Finally, $\tAdS$ is naturally endowed with a boundary $\partial_\infty\tAdS$, which is the covering of $\partial_\infty\AdS^3$ corresponding to the cyclic subgroup of $\pi_1(\partial_\infty\AdS^3)$ generated by a meridian. We will thus extend  $p:\USL\to\PSLR$ to the universal covering map:
\begin{equation} \label{eq extending p}
p:(\tAdS\cup\partial_\infty\tAdS)\to(\AdS^3\cup\partial_\infty\AdS^3)~,
\end{equation}
which we still denote $p$ by a small abuse of notation.

\subsection{Lifting representations}

We can now prove the following lemma, which will also serve to define the aforementioned \emph{standard lift} of a representation $\rho$.

\begin{lemma} \label{lemma lifting}
Let $\rho_l,\rho_r:\pi_1(S)\to\PSLR$ be Fuchsian representations. Then there exists a unique lift
$$\widetilde\rho:\pi_1(S)\to\isom_0(\tAdS)$$
of the representation
$$\rho=(\rho_l,\rho_r):\pi_1(S)\to\isom_0(\AdS^3)$$
with the following property. Given any
spacelike $\rho$-equivariant embedding $\sigma:\widetilde S\to\AdS^3$, every lift $\widetilde\sigma_i:\widetilde S\to \tAdS$ of $\sigma$ (where $i\in\Z$) is $\widetilde\rho$-equivariant.
\end{lemma}

\begin{remark}
The existence of \emph{some} lift of $\rho$ to $\isom_0(\tAdS)$ can be easily proved by some topological argument, which we now sketch. Recall that we can associate every representation $\rho:\pi_1(S)\to G$, for any Lie group $G$, with an obstruction class 
$$e(\rho)\in H^2(S,\pi_1(G))~,$$
such that $e(\rho)=0$ if and only if $\rho$ is liftable to a representation of $\pi_1(S)$ into the universal cover $\widetilde G$. More generally, given a covering map $q:H\to G$, the representation $\rho$ is liftable to $\widetilde \rho:\pi_1(S)\to H$ if and only if $e(\rho)$ is in the image of
$$q_*:H^2(S,\pi_1(H))\to H^2(S,\pi_1(G))~.$$

In our case, $G=\isom_0(\AdS^3)\cong \PSL_2\R\times \PSL_2\R$ and $H=\isom_0(\tAdS)\cong (\USL\times\USL)/\Delta_{\mathcal Z}$. Hence $e(\rho)=(e(\rho_l),e(\rho_r))$, where $e(\rho_l)$ and $e(\rho_r)$ are the usual Euler classes of the  representations $\rho_l,\rho_r:\pi_1(S)\to \PSL_2\R$. On the other hand the image of $\pi_1(H)$ in $\pi_1(G)$ is the diagonal subgroup in $\mathcal Z\times\mathcal Z$. So $\rho$ is liftable to $\isom_0(\tAdS)$ if and only if $e(\rho_l)=e(\rho_r)$.

In particular, in the case considered in this paper, the Euler class of both representations $\rho_l$ and $\rho_r$ is maximal and thus $\rho$ is always liftable as in Lemma \ref{lemma lifting}. However, the statement of Lemma \ref{lemma lifting} provides a favourite lifting $\widetilde \rho$, which will be indeed called \emph{standard}, satisfying a geometric property necessary for our construction.
\end{remark}

\begin{proof}[Proof of Lemma \ref{lemma lifting}]

Let $\Gamma$ be the curve in $\partial_\infty\AdS^3=\RP^1\times\RP^1$ which is the graph of the unique orientation-preserving homeomorphism of $\RP^1$ which conjugates $\rho_l$ to $\rho_r$, and let $\Gamma_i$ be the preimages of $\Gamma$ by the covering map $p:(\tAdS\cup\partial_\infty\tAdS)\to(\AdS^3\cup\partial_\infty\AdS^3)$. That is,
$$p^{-1}(\Gamma)=\bigsqcup_{i\in\Z}\Gamma_i~,$$
where we recall from \eqref{eq extending p} that we use the letter $p$ to denote the extended projection map, by a small abuse of notation. 
As $\Gamma$ is homotopic to $p_*(c_0)$, where the loop $c_0$ represents the generator 
of $\pi_1(\partial_\infty\tAdS)$, it turns out that $p|_{\Gamma_i}:\Gamma_i\to\Gamma$ is 1-to-1 and $\mathcal Z$ acts freely and transitively on $\{\Gamma_i:i\in\Z\}$.

Now, fix a lift $\Gamma_{i_0}$. Then, for every element $\tau$ of the fundamental group $\pi_1(\AdS^3)$, there exists a unique lift of $\rho(\tau)$ to $\isom_0(\tAdS)$ which preserves $\Gamma_{i_0}$. Define $\widetilde \rho(\tau)$ this element. (In fact, it suffices to pick any lift, which will send $\Gamma_{i_0}$ to some $\Gamma_i$, and then compose with an element of the center.) It is easy to check that $\widetilde \rho$ defined in this way is a representation. 

Actually, $\widetilde \rho$ does preserve every lift $\Gamma_i$. In fact, if $\widetilde\sigma_i=\xi\circ\widetilde\sigma_0$ where $\xi\in\mathcal Z$, the unique representation which preserves $\Gamma_i$ is $\xi\circ\widetilde\rho\circ\xi^{-1}$. But since $\xi$ is in the center, this coincides with $\widetilde \rho$.

Let us now show that the representation $\widetilde \rho$ defined in this way satisfies the claimed property. If $\sigma$ is an equivariant embedding as in the hypothesis, as explained in Subsection \ref{subsec embeddings} (see \cite[Proposition 20]{Mess}), the embedding $\sigma$ extends to the visual boundary $\partial_\infty\widetilde S$, with $\partial_\infty\sigma(\widetilde S)=\Gamma$.

For each $i$ there exists a unique lift $\widetilde \sigma_i$ of $\sigma$ such that $\partial_\infty \widetilde\sigma_i(\widetilde S)=\Gamma_i$.
Since every curve $\Gamma_i$ is left invariant by $\widetilde\rho(\tau)$, for every $\tau\in\pi_1(S)$, the same holds for $\sigma_i(\widetilde S)$. Hence the embeddings $\widetilde \sigma_i$ are $\widetilde\rho$-equivariant, for every $i\in \Z$. This concludes the proof.
\end{proof}

We will call \emph{standard lift} of $\rho$ the representation $\widetilde \rho$ constructed in Lemma \ref{lemma lifting}:

\begin{defi} \label{defi standard lift}
Let $\rho_l,\rho_r:\pi_1(S)\to\PSLR$ be Fuchsian representations. We call \emph{standard lift} of $\rho=(\rho_l,\rho_r):\pi_1(S)\to\isom_0(\AdS^3)$
the representation
$$\widetilde\rho:\pi_1(S)\to\isom_0(\tAdS)$$
which leaves every connected component of $p^{-1}(\Gamma)$ invariant, where $\Gamma\subset\partial_\infty\AdS^3\cong \RP^1\times\RP^1$ is the graph of the the unique  orientation-preserving homeomorphism of $\RP^1$ which conjugates $\rho_l$ with $\rho_r$.
\end{defi}

\begin{remark} \label{remark other lifts of reps}
Observe that, given a representation $\rho=(\rho_l,\rho_r)$ into $\isom_0(\AdS^3)$, any other lift $\widetilde \rho'$ into $\isom_0(\tAdS)$ is of the form
\begin{equation} \label{eq other lifts}
\widetilde \rho'(\tau)=\xi(\tau)\widetilde \rho(\tau)~,
\end{equation}
where $\xi:\pi_1(S)\to\mathcal Z$ is a representation in the center $\mathcal Z$. In fact, we know that the kernel of $p:\USL\to\PSLR$ is $\mathcal Z$, hence the form \eqref{eq other lifts}. Moreover,  the condition that $\widetilde \rho'$ is a representation coincides exactly with the condition that $\xi$ is a representation with values in $\mathcal Z$, which is thus identified to an element of $\Hom(\pi_1(S),\Z)\cong H^1(S,\Z)$. 
\end{remark}

\subsection{Unit future timelike tangent bundle} \label{unit tangent bundle}

We will consider embeddings of surfaces in the \emph{unit future timelike tangent bundle} of $\tAdS$. Let us fix a time-orientation on $\AdS^3$, and thus on $\tAdS$. This is the subbundle $$\Pi:\TA\to\tAdS$$
 of the tangent bundle $T\tAdS$ whose fiber over the point $\gamma\in\tAdS$ is
 $$\Pi^{-1}(\gamma)=\{(\gamma,u)\in T\tAdS\,:\,\langle u,u\rangle=-1\text{ and }u\text{ is positively time-oriented}\}~,$$
 where $\langle\cdot,\cdot\rangle$ denotes the metric we defined on $\tAdS$. Given a point $(\gamma,u)$, we will denote by $\mathcal V$ the \emph{vertical bundle}, namely the subbundle of $T\TA$ defined by:
 $$\mathcal V_{(\gamma,u)}=\mathrm{Ker}\left(\Pi_*:T_{(\gamma,u)}\TA\to T_\gamma \tAdS \right)~.$$
That is, $\mathcal V_{(\gamma,u)}$ is the vector subspace tangent to the fiber (which has dimension 2). This is identified to the subspace of $T_{(\gamma,u)}\tAdS$ tangent to the set of unit positive timelike vectors, namely, to the orthogonal complement $u^\perp$ in $T_\gamma \tAdS$. 

On the other hand, given a vector $w\in T_\gamma\tAdS$, we define its \emph{horizontal lift} at the point $(\gamma,u)$ in the following way. Let $\gamma(t)$  be a curve in  $\tAdS$ such that $\gamma(0)=\gamma$ and $\gamma'(0)=w$. Let $u(t)$ be the parallel transport of $u(0)$ along $\gamma(t)$ with respect to the Levi-Civita connection $\nabla$ of $\tAdS$: that is, $u(t)$ satisfies $\nabla_{\gamma'(t)}u(t)=0$.  Observe that, by compatibility of the metric, $(\gamma(t),u(t))$ is in $\TA$. Then 
$$w^h:=(\gamma,u)'(0)$$ is the horizontal lift of $w$ at $(\gamma,u)$. 
Observe that the map
$$w\in T_\gamma\tAdS\to w^h$$
is linear and injective. Hence 
the \emph{horizontal subspace} is the (3-dimensional) vector subspace 
$$\mathcal H_{(\gamma,u)}:=\{w^h\,:\,w\in T_\gamma\tAdS\}~.$$

We can now define a natural metric on $\TA$, known as the \emph{Sasaki metric}:
\begin{defi} \label{defi sasaki}
The Sasaki metric is the pseudo-Riemannian metric $g_S$ on $\TA$ defined, for $X_1,X_2\in T_{(\gamma,u)}\TA$, by:
$$
g_S(X_1,X_2):=\begin{cases}
\langle w_1,w_2\rangle & \text{if }X_1,X_2\in \mathcal H_{(\gamma,u)}, X_1=w_1^h,X_2=w_2^h ~;\\
 \langle v_1,v_2\rangle & \text{if }X_1,X_2\in \mathcal V_{(\gamma,u)}\text{ and they correspond to }v_1,v_2\in u^\perp ~;\\
 0 & \text{if }X_1\in \mathcal H_{(\gamma,u)}, X_2\in \mathcal V_{(\gamma,u)} \text{ or viceversa}~.
\end{cases}
$$
\end{defi}

The isometry group $\isom_0(\tAdS)$ induces an obvious action on $\TA$, by means of:
\begin{equation} \label{eq induced isometry}
\varrho_*(\gamma,u):=(\varrho(\gamma),\varrho_*(u))~,
\end{equation}
for any $\varrho\in\isom_0(\tAdS)$. This action on $\TA$ preserves the horizontal and vertical subspaces, in the sense that 
$$\varrho_* \mathcal H_{(\gamma,u)}=\mathcal H_{\varrho_*(\gamma,u)}\qquad\text{and}\qquad
\varrho_* \mathcal V_{(\gamma,u)}=\mathcal V_{\varrho_*(\gamma,u)}~,$$
and acts by isometries for the Sasaki metric.

\section{Principal $\R$-bundles and their curvature} \label{sec principal bundles}

In this section we introduce the principal $\R$-bundle
$\pi_\rho:P_\rho \to (S,h_l)\times (S,h_r)$, its connection form $\omega_\rho$, and study some of its properties, in relation with equivariant embeddings of $\widetilde S$ into $\TA$. Most remarkably, in Proposition \ref{eq curvature} we give an expression for the curvature form of $P_\rho$ only in terms of a natural symplectic form $\Omega_\rho$ on the base $S\times S$.

\subsection{$\R$-action by geodesic flow}
Let us consider again the unit future timelike tangent bundle $\TA$. There is a natural $\R$-action on $\TA$ given by the \emph{geodesic flow}. Namely, given $t\in\R$ and $(\gamma,u)\in\TA$, we define 
\begin{equation} \label{eq defi geod flow}
\varphi_t(\gamma,u):=(\gamma(t),\gamma'(t))~,
\end{equation}
where $\gamma(t)$ is the unique geodesic of $\tAdS$ such that $\gamma(0)=\gamma$ and $\gamma'(0)=u$. That is, $\gamma(t)$ satisfies $\nabla_{\gamma'(t)}\gamma'(t)=0$, where $\nabla$ is the Levi-Civita connection of $\tAdS$. In particular, $\langle \gamma'(t),\gamma'(t)\rangle=-1$ for all $t$, and thus \eqref{eq defi geod flow} is well-defined on $\TA$.

Recalling the bijection in \eqref{eq bij univ} between the space of timelike geodesics of $\tAdS$ and $\Hyp^2\times\Hyp^2$, which is given by 
$$(x,y)\in\Hyp^2\times\Hyp^2\mapsto p^{-1}(L_{x,y})~,$$
we can define a projection
$$\pi:\TA\to\Hyp^2\times \Hyp^2~,$$
by mapping $(\gamma,u)$ to the pair $(x,y)\in\Hyp^2\times\Hyp^2$ such that, if $\gamma(t):\R\to\tAdS$ is the geodesic of $\tAdS$ with $\gamma(0)=\gamma$ and $\gamma'(0)=u$, then
$$p(\{\gamma(t)\,:\,t\in\R\})=L_{x,y}~.$$

The key point for our construction is that $\pi$ provides $\TA$ with a principal $\R$-bundle structure.

\begin{lemma}
The bundle $\pi:\TA\to\Hyp^2\times \Hyp^2$ is a $\R$-principal bundle, where the $\R$-action on $\TA$ is given by the geodesic flow.
\end{lemma}
\begin{proof}
By construction, the $\R$-action preserves every fiber $\pi^{-1}(x,y)$. Moreover, as each geodesic of $\tAdS$ is a real line, the $\R$-action on each geodesic is free and transitive.
\end{proof}

Recall that the group 
$\isom_0(\tAdS)=(\USL\times\USL)/\Delta_{\mathcal Z}$ acts on $\TA$ by isometries, and on $\Hyp^2\times\Hyp^2$ by
$$(\alpha,\beta)\cdot (x,y)=(p(\alpha)(x),p(\beta)(y))~,$$
where $\alpha,\beta\in\USL$ and $p$ is the covering map $p:\USL\to\PSLR$. 
With these definitions, we have:

\begin{lemma} \label{lemma equivariance pi}
The projection $\pi:\TA\to\Hyp^2\times \Hyp^2$ is equivariant for the natural action of $\isom_0(\tAdS)=(\USL\times\USL)/\Delta_{\mathcal Z}$. That is, if $\pi(\gamma,u)=(x,y)$, then for every $(\alpha,\beta)\in \USL\times\USL$, 
$$\pi\left((\alpha,\beta)_*(\gamma,u)\right)=(p(\alpha)(x),p(\beta)(y))~.$$
\end{lemma}
\begin{proof}
The proof is direct consequence of the definitions and of Equation \eqref{eq transformation geodesics}.
\end{proof}

We will now derive an alternative expression for the projection $\pi$, which will be useful in the following. Before that, observe that there is a natural embedding
$$f:\Hyp^2\to\psl_2\R~.$$
To define $f$, for every point $x\in\Hyp^2$ let us denote $\mathcal R_{\theta,x}\in\PSLR$ the elliptic isometry of $\Hyp^2$ which fixes $x$ and is a counterclockwise rotation  around $x$ of angle $\theta$. Then define
\begin{equation} \label{eq defi f}
f(x)=\left.\frac{d}{d\theta}\right|_{\theta=0} \mathcal R_{2\theta,x}\in\psl_2\R~.
\end{equation}
By the arguments of \cite[Section 2]{ksurfaces}, $f$ is an isometric embedding, for the hyperbolic metric of $\Hyp^2$ and the Lorentzian metric on $\psl_2\R$ (which is $1/8$ times the Killing form). Moreover, $f$ is equivariant with respect to the action of $\PSLR$ (or of $\USL$)
on $\Hyp^2$, and the adjoint action on $\psl_2\R$:
$$f(\gamma(x))=\left.\frac{d}{d\theta}\right|_{\theta=0} \mathcal R_{2\theta,\gamma(x)}
=\left.\frac{d}{d\theta}\right|_{\theta=0} \left(\gamma\circ \mathcal R_{2\theta,x}\circ\gamma^{-1}\right)=\mathrm{Ad}(\gamma)f(x)~.$$

\begin{lemma} \label{lemma projection lie algebra}
Given any $(\gamma,u)\in\TA$,
\begin{equation} \label{eq projection lie algebra}
(f,f)\circ\pi(\gamma,u)=((R_{\gamma^{-1}})_* u,(L_{\gamma^{-1}})_* u)~,
\end{equation}
where $f$ is defined in Equation \eqref{eq defi f}.
\end{lemma}
\begin{proof}
We need to show that, if the timelike geodesic with initial data $(\gamma,u)$ is the geodesic $L_{x,y}$, then $f(x)=(R_{\gamma^{-1}})_* u$ and $f(y)=(L_{\gamma^{-1}})_* u$. 

Let us first prove Equation \eqref{eq projection lie algebra} when $\gamma=\mathrm{id}$. In this case, $\pi(\mathrm{id},u_0)=(x_0,x_0)$ for some $x_0\in\Hyp^2$, since the timelike geodesics through the identity all have the form $L_{x,y}$ with $x=y$, as a consequence of Equation \eqref{eq closed geodesic}. Hence we need to check that for every $u_0\in T_{\mathrm{id}}\tAdS=\psl_2\R$, if the 1-parameter group generated by $u_0$ is the geodesic $L_{x_0,x_0}$, then $f(x_0)=u_0$. This is exactly the definition of $f$ in Equation \eqref{eq defi f}.

Now, let us prove the general case. Given any $\gamma\in\USL$ and any $u\in T_{\gamma}\tAdS$, if we pick $u_0$ such that $u=(L_\gamma)_*(u_0)$, then
$$(\gamma,\mathrm{id})_*(\mathrm{id},u_0)=(\gamma,u)~.$$
(Here we think of $(\gamma,\mathrm{id})\in\USL\times\USL$ as an isometry of $\tAdS$, by means of Equation \eqref{eq isometry universal ads}.) According to Lemma \ref{lemma equivariance pi}, 
$$\pi(\gamma,u)=(\gamma,\mathrm{id})\cdot \pi(\mathrm{id},u_0)=(\gamma,\mathrm{id})\cdot (x_0,x_0)=(\gamma(x_0),x_0)~,$$
On the other hand, since by the previous case $f(x_0)=u_0$,
$$((R_{\gamma^{-1}})_* u,(L_{\gamma^{-1}})_* u)=(\mathrm{Ad}(\gamma)u_0,u_0)=(\mathrm{Ad}(\gamma)f(x_0),f(x_0))
=(f(\gamma(x_0)),f(x_0))$$
by using the equivariance of $f$. This concludes the proof.
\end{proof}

\subsection{A principal $\R$-connection} \label{subsec principal connection}

We will define a connection form on the principal bundle $\pi:\TA\to\Hyp^2\times \Hyp^2$. For this purpose, we need to use the following result:

\begin{theorem}[\cite{geodesicflow}] \label{thm geodesicflow}
The $\R$-action on $\TA$ by the geodesic flow is an action by isometries of the Sasaki metric.
\end{theorem}

Theorem \ref{thm geodesicflow} is equivalent to saying that the generator of the geodesic flow action is a Killing field. This generator is the vector field $\chi\in\Gamma^\infty(T\TA)$ of the form
\begin{equation} \label{generator chi}
\chi(\gamma,u)=u^h\in \mathcal H_{(\gamma,u)}\subset T_{(\gamma,u)}\TA~,
\end{equation}
under the identification $T_{(\gamma,u)}\TA\cong \mathcal H_{(\gamma,u)}\oplus \mathcal V_{(\gamma,u)}$ explained in Subsection \ref{unit tangent bundle}.

Let us now define the connection form of the principal bundle $\pi:\TA\to\Hyp^2\times \Hyp^2$:
\begin{defi} \label{defi connection form}
The connection form of $\pi:\TA\to\Hyp^2\times \Hyp^2$ is the 1-form $\omega\in\Omega^1(\TA,\R)$ defined by:
$$\omega=-g_S(\chi,\cdot)~,$$
where $\chi$ is the generator of the geodesic flow and $g_S$ is the Sasaki metric.
\end{defi}

It is easy to check that $\omega$ indeed defines a principal connection. For this purpose, first observe that in this case the tangent line to any fiber of $\pi$ is naturally identified to $\R$ (which is of course the Lie algebra of the group $\R$). Given a vector $X$ tangent to the fiber, $X$ is of the form $a\chi$ for some $a\in\R$, and then our identification maps $X$ to $a$. To check that $\omega$ is a principal connection, we need to show:
\begin{enumerate}
\item $\omega$ is $\R$-invariant. Since the adjoint action of $\R$ is trivial, this reduces to check that, for any $t\in\R$, $\varphi_t^*\omega=\omega$. This is indeed true since $\chi$ is $\varphi_t$-invariant and $\varphi_t$ is an isometry of $g_S$ by Theorem \ref{thm geodesicflow}:
$$\varphi_t^*\omega=\omega((\varphi_t)_*(\cdot))=-g_S(\chi,(\varphi_t)_*(\cdot))=-g_S(\chi,\cdot)=\omega~,$$
\item If $X=a\chi$ is tangent to the fiber, then $\omega(X)=a$. In fact, in this case
$$\omega(X)=-g_S(\chi,a\chi)=-a g_S(\chi,\chi)=a~,$$
by using the expression 
$\chi(\gamma,u)=u^h$ of Equation \eqref{generator chi} and the fact that $u$ is a unit vector, so that $g_S(\chi,\chi)=\langle u,u\rangle=-1$.
\end{enumerate}

In terms of Ehresmann connections, Definition \ref{defi connection form} means that the horizontal distribution of the principal bundle $\pi:\TA\to\Hyp^2\times \Hyp^2$ is given by the subspaces orthogonal to the orbits of the geodesic flow.

\subsection{Principal bundles over $S\times S$}

Let us now fix a representation 
$$\rho=(\rho_l,\rho_r):\pi_1(S)\to\PSLR\times\PSLR$$
 with $\rho_l,\rho_r$ Fuchsian, and consider the standard lift $\widetilde \rho:\pi_1(S)\to\isom_0(\tAdS)$ as in Definition \ref{defi standard lift}.
 Then for every $\tau\in\pi_1(S)$, $\widetilde \rho(\tau)\in \isom_0(\AdS^3)$ induces an isometry of $\TA$ as in Equation \eqref{eq induced isometry}. It turns out easily that the action of $\widetilde \rho(\pi_1(S))$:
 \begin{itemize}
 \item commutes with the $\R$-action of the geodesic flow;
 \item induces the action of $\rho_l(\pi_1(S))\times \rho_r(\pi_1(S))$ on $\Hyp^2\times\Hyp^2$;
 \item is free and properly discontinuous on $\TA$.
 \end{itemize}
In fact, the first point is direct consequence of $\widetilde\rho$ acting by isometries; the second follows from Lemma \ref{lemma equivariance pi}; the third point follows from the second point and the fact that $\rho_l(\pi_1(S))$ and $ \rho_r(\pi_1(S))$ act freely and properly discontinuously on $\Hyp^2$.
 
Therefore, the quotient $\TA/\widetilde \rho(\pi_1(S))$ has a structure of principal $\R$-bundle over the base
$$(\Hyp^2/\rho_l(\pi_1(S)))\times(\Hyp^2/\rho_r(\pi_1(S)))=(S,h_l)\times (S, h_r)~,$$
where we are identifying the quotient of $\Hyp^2$ by a Fuchsian group with a hyperbolic structure on $S$. Moreover, since $\isom_0(\tAdS)$ acts on $\TA$ by isometries of the Sasaki metric, the quotient bundle is also endowed with a connection form, induced by the connection form $\omega$ of Definition \ref{defi connection form}. Let us summarize this in a definition:
\begin{defi}
Given a representation $\rho=(\rho_l,\rho_r):\pi_1(S)\to\PSLR\times \PSLR$, we define 
$$\pi_\rho:P_\rho \to (S,h_l)\times (S, h_r)$$
as the principal $\R$-bundle induced by $\pi:\TA\to\Hyp^2\times \Hyp^2$, endowed with the connection form $\omega_\rho$ induced by $\omega$, where $\Hyp^2/\rho_l(\pi_1(S))=(S,h_l)$ and 
$\Hyp^2/\rho_r(\pi_1(S))=(S,h_r)$.
\end{defi}

Recall from Lemma \ref{lemma lifting} that a $\rho$-equivariant spacelike embedding $\widetilde S$ into $\AdS^3$ can be lifted to a $\widetilde \rho$-equivariant embedding
into $\tAdS$, where $\widetilde \rho:\pi_1(S)\to\isom_0(\tAdS)$ is the standard lift. We will now observe that 
it also induces an equivariant embedding into $\TA$ which is orthogonal to the generator $\chi$ of the geodesic flow, and thus parallel for the connection form $\omega_\rho$.

\begin{lemma} \label{lemma lift embedding hor}
Given a spacelike $\widetilde \rho$-equivariant embedding $\widetilde \sigma:\widetilde S\to \tAdS$, let 
$$\widetilde \sigma_N:\widetilde S\to \TA~,$$
be the map defined by
$$\widetilde \sigma_N(x)=(\widetilde \sigma(x),N(x))~,$$
where $N(x)$ is the future-directed unit vector orthogonal to $d\widetilde\sigma(T_x\widetilde S)$. Then
\begin{itemize}
\item $\widetilde \sigma_N$ is an  embedding, equivariant for the action of $\pi_1(S)$ on $\TA$ induced by $\widetilde \rho$, tangent to the horizontal distribution defined by the connection form $\omega$;  
\item $\pi\circ \widetilde \sigma_N$ is an embedding into $\Hyp^2\times\Hyp^2$, equivariant for the action induced on $\Hyp^2\times\Hyp^2$ by $\rho$. 
\end{itemize}
\end{lemma}
\begin{proof}
The equivariance of $\widetilde \sigma_N$ and $\pi\circ \widetilde \sigma_N$ is a direct consequence of the definitions. It remains to check that for every point $x\in\widetilde S$, the image of the differential of $\widetilde \sigma_N$ at $x$ is in the kernel of $\omega$, or in other words, is orthogonal to $\chi$. Recall that, in the decomposition 
$$T_{(\widetilde\sigma(x),N(x))}\TA\cong \mathcal H_{(\widetilde\sigma(x),N(x))}\oplus \mathcal V_{(\widetilde\sigma(x),N(x))}~,$$
we have $\chi_{(\widetilde\sigma(x),N(x))}=N(x)^h$, while
$$d\widetilde\sigma_N(\dot x)=d\widetilde\sigma(\dot x)^h\oplus \nabla_{d\widetilde\sigma(\dot x)}N~.$$
Hence from the form of the Sasaki metric (Definition \ref{defi sasaki}), we have
$$g_S(\chi,d\widetilde\sigma_N(\dot x))=\langle N(x),d\widetilde\sigma(\dot x)\rangle =0$$
since $N(x)$ is the normal vector of the image of $\widetilde\sigma$ at $x$. This concludes the proof of the first point.

For the second point, $\pi\circ \widetilde \sigma_N$ is an immersion as a consequence of the first point. Moreover, by \cite[Lemma 6]{Mess} $\widetilde \sigma(\widetilde S)$ intersects the orbits of the geodesic flow only in one point, and thus $\pi\circ \widetilde \sigma_N$ is globally injective.
\end{proof}
Hence every $\rho$-equivariant spacelike embedding $\widetilde \sigma:\widetilde S\to\AdS^3$ gives rise to an embedding of $\Lambda:S\to S\times S$ isotopic to the diagonal, and to a parallel section of $P_\rho|_{\Lambda(S)}$. We now provide an example of parallel equivariant embeddings into $\TA$, giving rise to a section of $P_\rho$ over the diagonal in $S\times S$, which is not obtained in this way. 

\begin{example} \label{ex fuchsian}
Let $\rho_0:\pi_1(S)\to\PSLR$ be a Fuchsian representation, and let 
$$\rho=(\rho_0,\rho_0):\pi_1(S)\to\PSLR\times\PSLR~.$$
Then $\rho(\tau)$ fixes the identity of $\PSLR$ for every $\tau$, by  \eqref{eq isometry action}.
Now let us consider $f:\Hyp^2\to\psl_2\R=T_{\mathrm{id}}\tAdS$, which was defined in Equation \eqref{eq defi f}, and define:
$$\widetilde \sigma(x):=(\mathrm{id},f(x))~.$$
By the equivariance of $f$, $\widetilde \sigma$ gives a $\rho$-equivariant embedding of $\Hyp^2$ into $\TA$, such that its composition with the projection $\Pi:\TA\to\tAdS$ gives a constant map $\Pi\circ \widetilde \sigma(x)=\mathrm{id}$. 

It is straightforward from the definition of the Sasaki metric, that $\widetilde \sigma$ is orthogonal to the generator $\chi$ of the geodesic flow. Moreover for every $x$, $\widetilde \sigma(x)$ belongs to the orbit of the geodesic flow which corresponds to the geodesic $L_{x,x}$. Hence 
$$\pi\circ \widetilde \sigma (x)=(x,x)\in\Hyp^2\times\Hyp^2~.$$
Therefore $\widetilde \sigma$ induces a parallel section of the restriction of the bundle $P_{\rho}$ over the diagonal in $S\times S$.

However, let us remark that, for every $t\in(-\pi,\pi)$, 
$$\pi\circ \varphi_{2t}\circ \widetilde \sigma(x)=\exp(tf(x))~.$$
We observe that  the section $\varphi_{2t}\circ \widetilde \sigma$ is still orthogonal to $\chi$ and thus parallel. (In general, applying $\varphi_t$ to a surface of the form $\widetilde \sigma_N(\widetilde S)$ corresponds essentially to acting by the normal flow.) Hence  $\varphi_t\circ \widetilde \sigma$
is an embedding into $\TA$, such that $\Pi\circ \varphi_t\circ \widetilde \sigma$ is also an embedding into $\tAdS$. In particular, for $t=\pm\pi/2$, $\varphi_{\pm\pi/2}\circ \widetilde \sigma$ is an isometric embedding of $\Hyp^2$, with image a totally geodesic plane.


\end{example}

\subsection{Relation between curvature and symplectic form}

Let us now consider the local geometry of the principal bundle $\pi_\rho:P_\rho\to (S, h_l)\times (S, h_r)$. The main result of this section is an expression of the curvature of this bundle, which is a $\R$-valued 2-form, in terms of a natural symplectic form on the base.

The curvature of the principal bundle $\pi_\rho:P_\rho\to S\times S$ is a real-valued 2-form $R_\rho\in\Omega^2(P_\rho,\R)$. Given a vector $X\in T_p P_\rho$, let us denote its horizontal-vertical decomposition, according to the connection form $\omega_\rho$, as
$$X=h(X)+v(X)~,$$
where $\omega_\rho(h(X))=0$ and $\omega_\rho(v(X))=v(X)$. Then by the structure equation for principal bundles, the curvature $R_\rho$
 can be expressed as
\begin{equation} \label{eq curvature}
R_\rho=d\omega_\rho~,
\end{equation}
since the term $\omega_\rho\wedge\omega_\rho$ vanishes in this case. On the other hand, using Lemma \ref{lemma equivariance pi}, the base $S\times S$ of $P_\rho$ is naturally endowed with:
\begin{itemize}
\item The Riemannian metric $h_l\oplus h_r$, which is induced by the metric of $\Hyp^2\times \Hyp^2$.
\item The symplectic form $\Omega_\rho=p_l^*\Omega_l-p_r^*\Omega_r$, where $p_l,p_r:S\times S\to S$ are the left and right projections, and $\Omega_l$ and $\Omega_r$ are the area forms of $h_l$ and $h_r$, respectively.
\end{itemize}
\begin{prop} \label{prop curvature}
Given any representation $\rho=(\rho_l,\rho_r):\pi_1(S)\to\PSLR\times\PSLR$ with $\rho_l,\rho_r$ Fuchsian, the curvature of $\pi_\rho:P_\rho\to S\times S$ is:
$$R_\rho=\frac{1}{2}\pi_\rho^*\Omega_\rho~,$$
where $\Omega_\rho=p_l^*\Omega_l-p_r^*\Omega_r$ and $\Omega_l,\Omega_r$ are the area forms induced on $S$ by the area form of $\Hyp^2$.
\end{prop}
For the proof of Proposition \ref{prop curvature}, we will use two foliations of $\TA$ whose leafs are three-dimensional, namely the \emph{left-invariant} foliation $\cF^L=\{\cF^L_u:u\in f(\Hyp^2)\}$ and \emph{right-invariant} foliation $\cF^R=\{\cF^R_u:u\in f(\Hyp^2)\}$, with leaves parameterized by the image of the isometric embedding  $f:\Hyp^2\to \psl_2\R$. Let us define the left-invariant foliation. Consider, for any fixed future unit vector $u\in\psl_2\R$, the section $\Sigma^L_u:\tAdS\to\TA$ of the bundle $\Pi:\TA\to \tAdS$:
$$g\in\tAdS\mapsto \Sigma_u^L(g):= (g,(L_g)_*(u))~,$$
where $L_g$ is the left multiplication for the Lie group structure of $\tAdS$. Then define 
$$\cF^L_u=\Sigma_u^L(\tAdS)~.$$
Observe that $(\mathrm{id},u)\in\cF^L_u$ and that the submanifolds $\cF^L_u$ foliate $\TA$ as $u$ varies in $\Hyp^2$, identified to the subset of future normal vectors in $\psl_2\R$ by the usual isometric embedding $f:\Hyp^2\to\psl_2\R$ (see Equation \eqref{eq defi f}).

Analogously, we define $\Sigma_u^R(g):= (g,(R_g)_*(u))$ and
$$\cF^R_u=\Sigma_u^R(\tAdS)~.$$
Let us moreover denote $\mathcal D^L_{(\gamma,u)}$ and $\mathcal D^R_{(\gamma,u)}$ the tangent distributions of the foliations $\cF^L$ and $\cF^R$ respectively, through the point
$(\gamma,u)$.


\begin{lemma} \label{lemma generate tangent}
The intersection of two leaves $\cF^L_u$ and $\cF^R_{u'}$ is an orbit of the geodesic flow on $\TA$. 
In particular, given any $(\gamma,u)\in\TA$,
$$\cD^L_{(\gamma,u)}\cap \cD^R_{(\gamma,u)}=\mathrm{Span}(\chi(\gamma,u))~,$$
and therefore
$$T_{(\gamma,u)}\TA=\cD^L_{(\gamma,u)}+\cD^R_{(\gamma,u)}~.$$
\end{lemma}
\begin{proof}
First of all, we show that it suffices to prove the first claim when $u=u'$. In fact, it can be easily checked that, for any isometry of the form $(\alpha,\beta)\in\USL\times\USL$,
$$(\alpha,\beta)_* \Sigma^L_u(g)=\Sigma^L_{\mathrm{Ad}\beta(u)}(\alpha g \beta^{-1})~,$$
and therefore $(\alpha,\beta)_* \cF^L_u=\cF^L_{\mathrm{Ad}\beta(u)}$.
Similarly, $(\alpha,\beta)_* \cF^R_u=\cF^R_{\mathrm{Ad}\alpha(u)}$.
Hence, by applying the action of an element of $\isom_0(\TA)$, which clearly maps orbits of the geodesic flow to orbits of the geodesic flow, we can reduce to the case $u=u'$.

Now, observe that $(g,v)\in \cF^L_u\cap \cF^D_u$ if and only if $v=(L_g)_*(u)=(R_g)_*(u)$. Hence $\mathrm{Ad}g(u)=u$, and therefore $g$ commutes with the 1-parameter elliptic subgroup $g_t:=\{\exp(tu):t\in\R\}$. Thus $g=g_t$ for some $t$. This shows that $\cF^L_u\cap \cF^D_u$ consists precisely of the orbit of the geodesic flow through $(\mathrm{id},u)$. 

Since the orbits of the geodesic flow are generated by $\chi$, it follows that in terms of tangent distributions,
$$\cD^L_{(\gamma,u)}\cap \cD^R_{(\gamma,u)}=\mathrm{Span}(\chi(\gamma,u))~,$$
and in particular the planes  $\cD^L_{(\gamma,u)}$ and $\cD^R_{(\gamma,u)}$ generate the tangent space of $\TA$ at the point $(\gamma,u)$.
\end{proof}


\begin{proof}[Proof of Proposition \ref{prop curvature}]
As the statement has a local nature, we can work with the bundle $\pi:\TA\to\Hyp^2\times \Hyp^2$ in a neighborhood of a point $(\gamma,u)$. Moreover, using Lemma \ref{lemma equivariance pi} and the fact that the action of $\isom_0(\tAdS)$ preserves the connection $\omega$, and the action of $\PSLR\times\PSLR$ preserves the symplectic form $\Omega=p_l^*(d\mathrm{A}_{\Hyp^2})-p_r^*(d\mathrm{A}_{\Hyp^2})$, we can assume that $\gamma=\mathrm{id}$.

Let $R\in\Omega^2(\TA,\R)$. Since $\mathcal D^L_{(\mathrm{id},u)}$ and $\mathcal D^R_{(\mathrm{id},u)}$ generate $T_{(\mathrm{id},u)}\TA$ by Lemma \ref{lemma generate tangent}, it suffices to check that 
$$2R(X,Y)=\Omega(d\pi(X),d\pi(Y))~,$$
when $X,Y$ are both in $\mathcal D^L_{(\mathrm{id},u)}$, both in $\mathcal D^R_{(\mathrm{id},u)}$, or $X\in \mathcal D^L_{(\mathrm{id},u)}$ and $Y\in \mathcal D^R_{(\mathrm{id},u)}$. 
\\

\paragraph{\textbf{Case 1}} Let us first suppose $X,Y\in \mathcal D^L_{(\mathrm{id},u)}$. Hence $X=(d\Sigma^L_u)_{\mathrm{id}}(v)$ and $Y=(d\Sigma_u^L )_{\mathrm{id}}(w)$ for some $v,w\in\psl_2\R$. 
Observe that, if we extend $v$ and $w$ to left-invariants vector fields $v^l,w^l$, then for every $g\in\tAdS$, 
$$\omega((d\Sigma_u^L)_{\mathrm{g}}(v^l))=g_S(\chi_{\Sigma_u^L(g)},(d\Sigma_u^L)_{\mathrm{g}}(v^l))=\langle (L_g)_*(u),(L_g)_*(v)\rangle~,$$
since, in the horizontal-vertical decomposition of $\TA$, $\chi_{(\gamma,u)}=u^h$, the horizontal component of $(d\Sigma_u^L)_{\mathrm{id}}(v^l)$ equals the horizontal lift of $v^l=(L_g)_*(v)$, and we applied the definition of the Sasaki metric (Definition \ref{defi sasaki}). By left-invariance of the Killing form, it follows that:
$$\omega((d\Sigma_u^L)_{\mathrm{g}}(v^l))=\langle u,v\rangle~.$$
In particular, $\omega((d\Sigma_u^L)_{\mathrm{g}}(v^l))$ is a constant function of $g\in\tAdS$. 

Let us decompose $v=v_0+\lambda u$ and $w=w_0+\mu u$ with $v_0,w_0\in u^\perp$. Then we have, using Cartan's formula:
\begin{equation} \label{eq comp curvature}
\begin{aligned}
R(X,Y)&=d\omega(d\Sigma_u^L(v),d\Sigma_u^L(w)) \\
&=d\Sigma_u^L(v).\omega(d\Sigma_u^L(w^l))-d\Sigma_u^L(w).\omega(d\Sigma_u^L(v^l))-\omega[d\Sigma_u^L(v^l),d\Sigma_u^L(w^l)] \\
&=-\omega(d\Sigma_u^L[v^l,w^l])=-\langle u,[v,w]_L\rangle=-2\langle u,v\boxtimes w\rangle \\
&=-2\langle u,v_0\boxtimes w_0\rangle=2(d\mathrm{A}_{\Hyp^2})(v_0,w_0)~.
\end{aligned}
\end{equation}
Here, from the second to the third line we used that $\omega(d\Sigma_u^L(v^l))$ and $\omega(d\Sigma_u^L(w^l))$ are constant functions. In the third line we used that the Lie bracket of left-invariant vector fields coincides with the bracket of Lie algebra, which is the same as the cross-product for the Lorentzian metric on the Lie algebra, up to a factor. Namely 
\begin{equation} \label{eq cross product bracket}
[v,w]_L=2v\boxtimes w~,
\end{equation} 
where $\boxtimes$ is the Lorentzian cross-product which is uniquely determined by the condition $\langle v\boxtimes w,u\rangle=d\mathrm{Vol}(v,w,u)$.
Equation \eqref{eq cross product bracket} is a consequence of the arguments explained in \cite[§2.1]{ksurfaces}.

On the other hand, by Lemma \ref{lemma projection lie algebra}, we have
$$(f,f)\circ \pi\circ \Sigma_u^L(g)=((R_{g^{-1}})_*(L_g)_*u,u)=(\mathrm{Ad}(g)u,u)\in f(\Hyp^2)\times f(\Hyp^2)\subset\psl_2\R\times \psl_2\R~.$$
Hence, recalling that $X=(d\Sigma_u^L)_{\mathrm{id}}(v)$ and $Y=(d\Sigma_u^L)_{\mathrm{id}}(w)$, we have
$$(df,df)\circ d\pi_{\mathrm{id,u}}(X)=([u,v]_L,0)=(2u\boxtimes v,0)=2(u\boxtimes v_0,0)~.$$
Since the embedding $f$ of $\Hyp^2$ in $\psl_2\R$ satisfies $df_u(Jv_0)=u\boxtimes v_0$,
where $J$ is the almost-complex structure of $\Hyp^2$, we have
\begin{equation} \label{eq proj left invariant case 1}
d\pi_{\mathrm{id,u}}(X)=2(Jv_0,0)~.
\end{equation}
Analogously
\begin{equation} \label{eq proj left invariant case 12}
d\pi_{\mathrm{id,u}}(Y)=2(Jw_0,0)~.
\end{equation} 
In conclusion, putting together Equations \eqref{eq comp curvature}, \eqref{eq proj left invariant case 1} and \eqref{eq proj left invariant case 12}, one obtains:
$$\pi^*\Omega(X,Y)=\Omega(d\pi_{\mathrm{id,u}}(X),d\pi_{\mathrm{id,u}}(Y))=4(d\mathrm{A}_{\Hyp^2})(Jv_0,Jw_0)=4(d\mathrm{A}_{\Hyp^2})(v_0,w_0)=2R(X,Y)~,$$
as claimed.
\\

\paragraph{\textbf{Case 2}}
When $X$ and $Y$ are in $\mathcal D^R_{(\mathrm{id},u)}$, that is $X=(d\Sigma_u^R)_{\mathrm{id}}(v)$ and $Y=(d\Sigma_u^R)_{\mathrm{id}}(w)$, the proof goes in a similar way. Extending $v$ and $w$ to right-invariant vector fields $v^r,w^r$, one has analogously 
$$\omega((d\Sigma_u^R)_{\mathrm{g}}(v^r))=\langle u,v\rangle~.$$
Therefore, with the same decomposition $v=v_0+\lambda u$ and $w=w_0+\mu u$, one shows
$$R(X,Y)=d\omega(d\Sigma_u^R(v),d\Sigma_u^R(w))
=-\omega(d\Sigma_u^R[v^r,w^r])=\langle u,[v,w]_L\rangle=-2(d\mathrm{A}_{\Hyp^2})(v_0,w_0)~,$$
where the difference in sign with respect to the previous case is due to the fact that, for right-invariant vector fields,
$[v^r,w^r]=-[v,w]_L$, where $[v,w]$ is the Lie algebra bracket, defined using left-invariant extensions.

On the other hand, in this case $(f,f)\circ \pi\circ \Sigma_u^R(g)=(u,\mathrm{Ad}(g^{-1})u)$, whence 
\begin{equation} \label{eq proj right invariant case 2}
d\pi_{\mathrm{id,u}}(X)=2(0,-Jv_0)\qquad\textrm{and}\qquad d\pi_{\mathrm{id,u}}(Y)=2(0,-Jw_0)~.
\end{equation}
 Thus 
$$\pi^*\Omega(X,Y)=-4(d\mathrm{A}_{\Hyp^2})(-Jv_0,-Jw_0)=-4(d\mathrm{A}_{\Hyp^2})(v_0,w_0)=2R(X,Y)~,$$
and this concludes the second case.
\\

\paragraph{\textbf{Case 3}}
Finally, if $X\in \mathcal D^L_{(\mathrm{id},u)}$ and $Y\in \mathcal D^R_{(\mathrm{id},u)}$, we know from Equation \eqref{eq proj left invariant case 1} that  $d\pi_{\mathrm{id,u}}(X)=2(Jv_0,0)$ and from Equation \eqref{eq proj right invariant case 2} that $d\pi_{\mathrm{id,u}}(Y)=2(0,-Jw_0)$, hence $\pi^*\Omega(X,Y)=0$.

When computing the curvature, in this case we can write $X=(d\Sigma_u^L)_{\mathrm{id}}(v)$ and $Y=(d\Sigma_u^R)_{\mathrm{id}}(w)$, and extend $v$ to a left-invariant $v^l$ and $w$ to a right-invariant $w^r$. Hence one has again 
\begin{align*}
R(X,Y)&=d\omega(d\Sigma_u^L(v),d\Sigma_u^R(w))  \\
&=d\Sigma_u^L(v).\omega(d\Sigma_u^L(w^r))-d\Sigma_u^R(w).\omega(d\Sigma_u^L(v^l))-\omega[d\Sigma_u^L(v^l),d\Sigma_u^R(w^r)] \\
&=-\omega[d\Sigma_u^L(v^l),d\Sigma_u^R(w^r)]=0~,
\end{align*}
since left-invariant and right-invariant vector fields commute. This concludes the proof.
\end{proof}

Recall that, given a symplectic manifold $(M,\Omega)$ with $\dim M=2n$, and an embedding $\Lambda:N\to M$ with $\dim N=n$, then $\Lambda$ is \emph{Lagrangian} if $\Lambda^*\Omega=0$. We will denote by $\Lambda^*P_\rho$ the pull-back bundle, that is, the bundle on $S$ defined by the following commutative diagram:
\begin{equation} \label{eq diagram pullback}
\xymatrix{
 \Lambda^*P_\rho \ar[r]^-{i} \ar[d] & P_\rho \ar[d]^-{\pi} \\
S \ar[r]^-{\Lambda} & S\times S \\
}
\end{equation}
Hence with this notation, the inclusion $i:\Lambda^*P_\rho\to P_\rho$ induces a connection form on $P_\rho$, which is again a $\R$-principal bundle.
From Proposition \ref{prop curvature} and the definition of Lagrangian embedding, we have:

\begin{cor} \label{cor lagrangian}
Let $\rho=(\rho_l,\rho_r):\pi_1(S)\to\PSLR\times\PSLR$ be a representation with $\rho_l,\rho_r$ Fuchsian, and let $\Lambda:S\to (S\times S,\Omega_\rho)$ be an embedding. Then $\Lambda$ is Lagrangian if and only if $(\Lambda^*P_\rho,i^*\omega_\rho)$ is flat.
\end{cor}

\begin{remark}
Recall that we have already produced several examples of $\rho$-equivariant embeddings $\widetilde\sigma:\widetilde S\to\TA$, with the property that the image of the differential of $\widetilde\sigma$ is in the horizontal distribution. For instance those given by lifting an equivariant embedding into $\tAdS$ (Lemma \ref{lemma lift embedding hor}), or in Example \ref{ex fuchsian}. 

By $\rho$-equivariance, $\pi\circ\widetilde\sigma$ induces an embedding (say, $\Lambda_{\widetilde \sigma}$) of $S$ into $S\times S$, where $S\times S$ is endowed with the symplectic form $\Omega_\rho$.  By Corollary \ref{cor lagrangian}, $\Lambda_{\widetilde \sigma}$ is a Lagrangian embedding. Hence $\widetilde \sigma$ induces a Lagrangian submanifold in $(S\times S,\Omega_\rho)$. In \cite{barbotkleinian} it has been already proved that the submanifold $\Lambda_{\widetilde \sigma}(S)$ is Lagrangian, by means of different arguments. In the present work, we will deal with the \emph{global} character of the problem, thus giving an obstruction to find a global parallel section over a given Lagrangian submanifold of $(S\times S,\Omega_\rho$). The \emph{local} theory was already clarified in \cite[§3]{barbotkleinian}.

As a particular case, suppose $\sigma:\widetilde S\to\AdS^3$, $\widetilde\sigma$ is its lift to $\tAdS$ and $\widetilde \sigma_N$ is the associated map into $\TA$. It is known from \cite{Schlenker-Krasnov} that, if $\sigma(\widetilde S)$ has curvature different from zero at every point, then the image of the embedding $\Lambda:S\to S\times S$ is a graph over $S$, that is $\Lambda(x)=(x,\varphi(x))$ for a diffeomorphism $\varphi$ isotopic to the identity. Moreover, $0=\Lambda^*\Omega=\Omega_l-\varphi^*\Omega_r$. Hence one recovers the already known result that $\varphi:(S,\Omega_l)\to(S,\Omega_r)$ is a symplectomorphism isotopic to the identity.
\end{remark}

\section{Flat subbundles and their holonomy} \label{sec holonomy}

In this section we study the symplectic geometry of $(S\times S,\Omega_\rho)$, by means of the flux map for Lagrangian embeddings $\Lambda:S\to (S\times S,\Omega_\rho)$, and we relate it to the holonomy of the flat principal $\R$-bundles $\Lambda^*P_\rho$.

\subsection{Lagrangian embeddings}

Let us consider the space of Lagrangian embeddings of $S$ into $(S\times S,\Omega_\rho)$ isotopic to the diagonal
$$\Delta:x\in S\mapsto (x,x)\in S\times S~.$$ 
That is, we define:
$$\mathcal M_\rho:=\{\Lambda:S\to (S\times S,\Omega_\rho)\,:\,\Lambda^*\Omega_\rho=0, \Lambda\sim \Delta\}~.$$
There is a right action on $\mathcal M_\rho$ of the group
$$\mathrm{Diff}_0(S):=\{\varphi:S\to S\text{ diffeomorphism},\varphi\sim\mathrm{id}\}~,$$
by pre-composition:
$$\mathcal M_\rho\times \mathrm{Diff}_0(S)\to \mathcal M_\rho \qquad (\Lambda,\varphi)\mapsto \Lambda\circ\varphi~.$$
Hence we define
$$\mathcal L_\rho:=\mathcal M_\rho/\mathrm{Diff}_0(S)~.$$
We say that a path $f_\bullet:[0,1]\to\mathcal L_\rho$ is smooth if there exists a smooth lift $\Lambda_\bullet:[0,1]\to\mathcal M_\rho$ such that $f_t=[\Lambda_t]$.

Now, let $\Lambda_\bullet:[0,1]\to\mathcal M_\rho$ be a piecewise smooth path of Lagrangian embeddings. Consider a loop $\ell:\mathbb{S}^1\to S$, with $\ell(1)=x_0$, where we adopted the notation $\mathbb{S}^1=\{e^{2\pi i s}:s\in [0,1]\}$. Let 
$$F_\ell:\mathbb{S}^1\times [0,1]\to (S\times S,\Omega_\rho)$$
be defined by
\begin{equation} \label{defi fell}
F_\ell(e^{2\pi i s},t)=\Lambda_t(\ell(e^{2\pi i s}))~.
\end{equation}
By \cite[Lemma 6.1]{MR3124936}, the integral
$$\int_{\mathbb{S}^1\times [0,1]}F_\ell^*\Omega_\rho=\int_{\mathbb{S}^1\times [0,1]}\Omega_\rho\left(\frac{dF_\ell}{ds},\frac{dF_\ell}{dt}\right)dsdt$$
only depends on the homotopy class of the loop $\ell$ and on the homotopy class of the path $[\Lambda_\bullet]:[0,1]\to \mathcal L_\rho$. This justifies the following definition:

\begin{defi} \label{defi flux for paths}
Given a piecewise smooth path $[\Lambda_\bullet]:[0,1]\to\mathcal L_\rho$, define
$$\Flux:C^\infty([0,1],\mathcal L_\rho)\to\Hom(\pi_1(S),\R)$$
as the function
$$\Flux([\Lambda_\bullet]):\tau\mapsto \int_{\mathbb{S}^1\times [0,1]}F_\ell^*\Omega_\rho~,$$
where $F_\ell(e^{2\pi i s},t)=\Lambda_t(\ell(e^{2\pi i s}))$ and $\tau=[\ell]\in\pi_1(S)$.
\end{defi}

The above construction is very general in symplectic geometry. However, as we shall see later, in our special case the map $\Flux$ does not even depend on the homotopy class of the path $\Lambda_\bullet$, but only on its endpoints $\Lambda_0$ and $\Lambda_1$.

\subsection{Variation of the holonomy}

By Corollary \ref{cor lagrangian}, given any $\Lambda\in\mathcal M_\rho$, the bundle $\Lambda^*P_\rho$ is a flat principal $\R$-bundle, that is, $\Lambda^*P_\rho$ admits trivializations with transition maps which are translations. Moreover, if $[\Lambda_1]=[\Lambda_2]\in \mathcal L_\rho$, then $(\Lambda_1)^*P_\rho$ and $(\Lambda_2)^*P_\rho$ are isomorphic as flat $\R$-bundles.

Let us recall the definition of \emph{holonomy} in this case.

\begin{defi} \label{defi holonomy}
Given an affine $\R$-bundle $E\to S$ with a flat connection $\omega$, the \emph{holonomy} is the representation 
$\hol_E:\pi_1(S)\to\R$
such that for every loop $\ell:\mathbb{S}^1\to S$ with $\ell(1)=x_0$, and for every section $p:[0,1]\to E$ such that $p(s)\in E_{\ell(e^{2\pi is})}$ and $\omega(p'(s))=0$, then $\hol_E([\ell])=t_0$ where $t_0\in\R$ is the unique value such that
$$p(0)=\varphi_{t_0}\circ p(1)~.$$
\end{defi}
In fact, due to the condition of flatness of the connection $\omega$, it turns out that the definition of $\hol_E:\pi_1(S)\to\R$ does not depend on the representative $\ell$ of an element $\tau\in\pi_1(S)$, nor on the initial point $p(0)$. We say that a flat $\R$-bundle $E\to S$ is \emph{trivial} if its holonomy is the trivial representation. Indeed, this is equivalent to saying that $E$ admits a parallel global section.

Observe that, since the structure group of flat $\R$-bundles is contractible, the bundles of the form $\Lambda^*P_\rho$ are always topologically trivial.

\begin{prop} \label{prop holonomy flat}
Let $\rho=(\rho_l,\rho_r):\pi_1(S)\to\PSLR\times\PSLR$ be a representation with $\rho_l,\rho_r$ Fuchsian, and let $\Lambda\in \mathcal M_\rho$. Let $\Sigma$ be any global section of $\Lambda^* P_\rho\to S$. Then
\begin{equation} \label{eq connection section}
\hol_{\Lambda^*P_\rho}(\tau)=\int_\tau \Sigma^*\omega_\rho~.
\end{equation}
\end{prop}
The fact that the right-hand side of Equation \eqref{eq connection section} does not depend on the choice of the global section $\Sigma$ follows from the following lemma, which will be also important in the proof. This is a consequence of the general formula for the transformation rule for the connection form on a principal bundle, see for instance \cite[Chapter 2]{kobnomi}, hence we omit the proof.

\begin{lemma} \label{lemma change connection form}
Let $S_0\subset S\times S$ be a submanifold. Suppose $\Sigma_1,\Sigma_2:S_0\to P_\rho$ are sections, and let $f:S_0\to\R$ be such that
$$\Sigma_2(x)=\varphi_{f(x)}\circ \Sigma_1(x)$$
for every $x\in S_0$. Then
$$\Sigma_2^*\omega_\rho=\Sigma_1^*\omega_\rho+df~.$$
\end{lemma}
Applying Lemma \ref{lemma change connection form} to $S_0=\Lambda(S)$, we obtain that the 1-forms $\Sigma_1^*\omega_\rho$ and $\Sigma_2^*\omega_\rho$ differ by a coboundary, and thus the right-hand side in Equation \eqref{eq connection section} does not depend on the choice of the section $\Sigma$. We can now complete the proof of Proposition \ref{prop holonomy flat}.
\begin{proof}[Proof of Proposition \ref{prop holonomy flat}]
Let $\Sigma$ be a global section of $\Lambda^*P_\rho\to S$. Let $\widetilde S$ be the universal cover of $S$, lift $\Lambda:S\to (S,h_l)\times (S,h_r)$ to the Lagrangian embedding $\widetilde \Lambda:\widetilde S\to\Hyp^2\times\Hyp^2$, and $\Sigma$ to a section $\widetilde \Sigma$ of the pull-back bundle of $\Lambda^*P_\rho$ over $\widetilde S$, which is identified to $\widetilde\Lambda^* P$. 

By Equation \eqref{eq curvature} and Proposition \ref{prop curvature}, we have 
$$d(\Sigma^*\omega_\rho)=\Sigma^*d\omega_\rho=\Sigma^* R_\rho=\Sigma^*(\pi_\rho)^*\Omega_\rho=\Omega_\rho~,$$
and therefore, as $\Lambda$ is a Lagrangian embedding,
$$d(\Lambda^*\Sigma^*\omega_\rho)=\Lambda^*\Omega_\rho=0~.$$
This means that the 1-form $\Lambda^*\Sigma^*\omega_\rho$ is closed, and thus also its lift $\widetilde\Lambda^*\widetilde\Sigma^*\omega$. Since $\widetilde S$ is simply connected, there exists a function $f_0:\widetilde S\to\R$ such that
$$\widetilde\Lambda^*\widetilde\Sigma^*\omega=-df_0~.$$
Now, consider the section $\Sigma_0$ of $\widetilde\Lambda^*P$ defined by
$$\Sigma_0(x)=\varphi_{(f_0\circ\widetilde\Lambda^{-1})(x)}\circ \widetilde\Sigma(x)~.$$
By Lemma \ref{lemma change connection form}, $\Sigma_0$ is a parallel section over $\widetilde\Lambda(\widetilde S)$. Hence, it follows from Definition \ref{defi holonomy} that 
$$\hol_{\Lambda^*P_\rho}(\tau)=f_0(x_0)-f_0(\tau(x_0))=-\int_{\widetilde\ell} df_0~,$$
where $\widetilde\ell:[0,1]\to \widetilde S$ is any path connecting $x_0$ to $\tau(x_0)$. Hence
$$\hol_{\Lambda^*P_\rho}(\tau)=\int_{\widetilde\ell} \widetilde\Lambda^*\widetilde\Sigma^*\omega=\int_\tau\Lambda^*\Sigma^*\omega_\rho~,$$
as claimed.
\end{proof}

We are now ready to show that, given two Lagrangian embeddings $\Lambda_0,\Lambda_1:S\to (S\times S,\Omega_\rho)$, the difference between the holonomy of $\Lambda_0^*P_\rho$
and $\Lambda_1^*P_\rho$ coincides precisely with the value of the map $\Flux$ on a path of Lagrangian embeddings connecting $\Lambda_0$ and $\Lambda_1$.

\begin{prop} \label{prop holonomy flux}
Let $\Lambda_\bullet:[0,1]\to\mathcal M_\rho$ be a smooth path of Lagrangian embeddings. Then
$$\hol_{\Lambda_1^*P_\rho}=\hol_{\Lambda_0^*P_\rho}+\Flux([\Lambda_\bullet])~.$$
\end{prop}
\begin{proof}
Let $\Sigma:S\times S\to P_\rho$ be a global section, and define $\Sigma_t=\Sigma\circ \Lambda_t:S\to \Lambda_t^* P_\rho$. From Proposition \ref{prop holonomy flat}, we have:
\begin{equation} \label{eq difference holonomy integrals}
\hol_{\Lambda_1^*P_\rho}(\tau)-\hol_{\Lambda_0^*P_\rho}(\tau)=\int_\tau \left(\Sigma_1^*\omega_\rho- \Sigma_0^*\omega_\rho\right)~.
\end{equation}
Now, let us consider the map
$$F:S\times [0,1]\to (S\times S,\Omega_\rho) \qquad F(x,t)=\Lambda_t(x)~.$$
Consider the 1-form $\eta_\rho=F^*\Sigma^*\omega_\rho$ on $(S\times [0,1])$. Then we have, for any $v\in T_x S$,
\begin{align*}
(\Sigma_1^*\omega_\rho- \Sigma_0^*\omega_\rho)(v)&=\int_0^1 (\mathcal L_{\frac{\partial}{\partial t}}\eta_\rho)_{(x,t)}(v) \\
&=\int_0^1 (d\eta_\rho)_{(x,t)}\left(\frac{\partial}{\partial t},v\right)-d\left(\eta_\rho\left(\frac{\partial}{\partial t}\right)\right)(v)~.
\end{align*}
(In the first equality above, $\mathcal L$ denotes the Lie derivative.) Hence we have
\begin{equation} \label{eq integrals}
\int_\tau \left(\Sigma_1^*\omega_\rho- \Sigma_0^*\omega_\rho\right)=\int_\tau \int_0^1 (d\eta_\rho)_{(x,t)}\left(\frac{\partial}{\partial t},\cdot\right)~,
\end{equation}
since the last term in the previous equation is an exact 1-form and thus its periods vanish.

 Now, recall from Proposition \ref{prop curvature} that the curvature of $P_\rho$ is $R_\rho=d\omega_\rho=\pi^*\Omega_\rho$, where $\omega_\rho$ is the connection form of $P_\rho$, and thus $\Omega_\rho=\Sigma^*d\omega_\rho$. 
 Therefore
\begin{equation} \label{eq compare pullback}
d\eta_\rho=d(F^*\Sigma^*\omega_\rho)=F^*\Sigma^*(d\omega_\rho)=F^*\Omega_\rho~.
\end{equation}
 
Finally, if $\ell:[0,1]\to S$ is a representative of $\tau\in\pi_1(S)$, the restriction of $F$ to the image of $\ell$ coincides with the function $F_\ell$ defined in Equation \eqref{defi fell}:
$$F(\ell(e^{2\pi i s}),t)=F_\ell(e^{2\pi i s},t)$$
 Hence, applying this fact and Equation \eqref{eq compare pullback} inside Equation \eqref{eq integrals}, we obtain
$$\int_\tau \left(\Sigma_1^*\omega_\rho- \Sigma_0^*\omega_\rho\right)=\int_{\ell(\mathbb{S}^1)\times [0,1]}F^*\Omega_\rho=\int_{\mathbb{S}^1\times [0,1]}F_\ell^*\Omega_\rho=\Flux([\Lambda_\bullet])~,$$
where in the last step we applied Definition \ref{defi flux for paths}. By Equation \eqref{eq difference holonomy integrals}, this concludes the proof.
\end{proof}

\subsection{Orbits of Hamiltonian diffeomorphism group}

In this subsection we collect some results on the flux homomorphism in relation with symplectic geometry, which will be relevant in the proof of our main results. We won't give any proof of the stated results in this subsection; references are provided.

First, let us remark that the flux homomorphism was defined classically (by Calabi in \cite{MR0350776}, see also \cite{MR490874}) on the universal cover of the group of symplectomorphisms of a symplectic manifold. To be precise, if $(M,\Omega_M)$ is a symplectic manifold, let 
$\Symp_0(M,\Omega_M)$ be the connected component of the identity in the group of symplectomorphisms of $(M,\Omega_M)$, and let 
$[\Phi_\bullet]\in \widetilde\Symp_0(M,\Omega)$ be the homotopy class of a smooth path 
$$\Phi_\bullet:[0,1]\to\Symp(M,\Omega)~,$$ with $\Phi_0=\mathrm{id}$. Then the flux homomorphism introduced by Calabi is a map:
$$\CFlux:\widetilde\Symp_0(M,\Omega)\to H^1_\dR(M,\R)$$
defined in the following way. Let $\xi_t$ be the generating vector field of $\Phi_t$, namely
\begin{equation} \label{eq generating}
\xi_{t_0}(x)=\left.\frac{d}{dt}\right|_{t=t_0}\Phi_t\circ \Phi_{t_0}^{-1}(x)~.
\end{equation}
Then
$$\CFlux([\Phi_\bullet])=\int_{[0,1]}\Omega_M(\xi_t,\cdot)dt\in H^1_\dR(M,\R)~.$$
It turns out again that $\CFlux$ only depends on the homotopy class of $\Phi_\bullet$, and that it is a group homomorphism.

\begin{remark}
If $(M,\Omega_M)=(S\times S,p_l^*\Omega_S-p_r^*\Omega_S)$, where $\Omega_S$ is a symplectic form on $S$, then it is easy to check that:
\begin{itemize}
\item A diffeomorphism $\Phi:(S,\Omega_S)\to (S,\Omega_S)$ is symplectomorphism if and only if its graph is a Lagrangian submanifold of $(M,\Omega_M)$;
\item If we denote $\Lambda_t$ the graph of $\Phi_t$, or more precisely
$$\Lambda_t(x)=(x,\Phi_t(x))\in S\times S~,$$
 then
$$\Flux([\Lambda_\bullet])(\tau)=\int_\tau\CFlux([\Phi_\bullet])$$
for every $\tau\in\pi_1(S)$. 
\end{itemize}
Hence for those Lagrangian submanifolds of $(M,\Omega)$ which are graphs of symplectomorphism, the two definitions of flux coincide --- up to the standard identification of $H^1_\dR(S,\R)$ with $\Hom(\pi_1(S),\R)$.
\end{remark}

Recall the following definition of Hamiltonian isotopies and symplectomorphisms:
\begin{defi}
A \emph{Hamiltonian isotopy} is a path $\Phi_\bullet:[0,1]\to\Symp(M,\Omega_M)$ for which there exist \emph{Hamiltonian functions} $H_t:M\to\R$ for every $t$ such that 
$$dH_t=\Omega_M(\xi_t,\cdot)~,$$
where $\xi_t$ is the generating vector field as in Equation \eqref{eq generating}. We say that two symplectomorphisms $\Phi_1,\Phi_2:(M,\Omega_M)\to (M,\Omega_M)$ are \emph{Hamiltonian isotopic} is there exists a Hamiltonian isotopy connecting them. A symplectomorphism $\Phi:(M,\Omega_M)\to (M,\Omega_M)$ is \emph{Hamiltonian} if it is Hamiltonian isotopic to the identity. Finally, we denote $\Ham(M,\Omega_M)$ the group of Hamiltonian symplectomorphisms.
\end{defi}

The bridge between the two definitions of flux is covered by the following fact, which is proved in \cite[Lemma 6.6]{MR3124936}. Here (and in the rest of the paper) we restrict to the case $(M,\Omega_M)=(S\times S,\Omega_\rho)$, which is the case of our interest.

\begin{lemma}[{\cite[Lemma 6.6]{MR3124936}}]
Let $[\Lambda_\bullet]:[0,1]\to\mathcal L_\rho$ be a smooth path of Lagrangian submanifolds, where $\Lambda_t:S\to (S\times S,\Omega_\rho)$. Then there exists a path of symplectomorphisms $\Phi_t:(S\times S,\Omega_\rho)\to (S\times S,\Omega_\rho)$ such that
$$[\Phi_\bullet\circ \Lambda_0]=[\Lambda_\bullet]\qquad\text{and}\qquad \Phi_0=\mathrm{id}~.$$
Moreover, for every $\tau\in\pi_1(S)$,
$$\Flux([\Lambda_\bullet])(\tau)=\int_\tau (\Lambda_0)^*\CFlux([\Phi_\bullet])~.$$
\end{lemma}

In fact, in \cite[Lemma 6.6]{MR3124936} the previous lemma is proved under the assumption that the restriction (pull-back) map
$$\Lambda_0^*:H^1_\dR(M,\R)\to H^1_\dR(S,\R)$$
is surjective, which is satisfied in our situation. Moreover, in the case under consideration here, $\Flux([\Lambda_\bullet])$ only depends on the endpoints $\Lambda_0$ and $\Lambda_1$ of the smooth path $\Lambda_\bullet$. This follows from Proposition \ref{prop holonomy flux}.

Hence also the following corollary applies (recalling that $\Ham(M,\Omega_M)$ denotes the group of Hamiltonian symplectomorphism):
\begin{cor}[{\cite[Corollary 6.8]{MR3124936}}] \label{cor flux vanish hamiltonian}
Two Lagrangian submanifolds $[\Lambda_0],[\Lambda_1]\in \mathcal L_\rho$ of $(S\times S,\Omega_\rho)$ 
are in the same orbit of the action of $\Ham(S\times S,\Omega_\rho)$ on $\mathcal L_\rho$ if and only if $\Flux([\Lambda_\bullet])=0$ for some (hence for every) smooth path $[\Lambda_\bullet]:[0,1]\to\mathcal L_\rho$ connecting them.
\end{cor}

\section{Proofs of the main results} \label{sec proof}

In this section we will conclude the proof of the main results stated in the Introduction. In particular, let us show Theorem \ref{main thm}:

\begin{reptheorem}{main thm}
Let $\rho=(\rho_l,\rho_r):\pi_1(S)\to\PSL_2\R\times \PSL_2\R$, where $\rho_l$ and $\rho_r$ are Fuchsian representations, and let $\widetilde\rho:\pi_1(S)\to\isom_0(\tAdS)$ be its standard lift. Then
$$\left\{\Lambda_{\widetilde\sigma}:\begin{aligned} \widetilde\sigma &\text{ is a }\widetilde\rho\text{-equivariant embedding orthogonal} \\
&\text{to the orbits of the geodesic flow} \end{aligned} \right\}=\Ham(S\times S,\Omega_\rho)\cdot\Lambda_\ML~,$$
where $\Lambda_\ML$ is the unique minimal Lagrangian submanifold of $(S\times S,\Omega_\rho)$ isotopic to the diagonal.
\end{reptheorem}

\begin{proof}
Recall that, from the definition of the connection form $\omega$ on $\pi:\TA\to\Hyp^2\times\Hyp^2$ (Definition \ref{defi connection form}), and the fact that ${image}(\pi\circ \widetilde\sigma)/\rho(\pi_1(S))$ is the submanifold $\Lambda_{\widetilde \sigma}\subset (S\times S,\Omega_\rho)$ associated to $\widetilde \sigma$, we obtain that  $\widetilde \sigma$ is orthogonal to the orbits of the geodesic flow if and only if it gives a parallel section of $\pi_\rho|_{\pi_\rho^{-1}(\Lambda_{\widetilde\sigma})}:{\pi_\rho^{-1}(\Lambda_{\widetilde\sigma})}\to \Lambda_{\widetilde\sigma}$ over $\Lambda_{\widetilde\sigma}$. 

Let us now show the two inclusions. First, let us observe that $\Lambda_\ML$ certainly can be obtained as $\Lambda_{\widetilde\sigma_{\mathrm{max}}}$, where $ \sigma_{\mathrm{max}}$ is the (unique) $\rho$-equivariant embedding of $\widetilde S$ into $\AdS^3$ of vanishing mean curvature, and ${\widetilde\sigma_{\mathrm{max}}}$ is its normal $\widetilde\rho$-equivariant lift into $\TA$. Now from any other $\Lambda$ in the $\Ham(S\times S,\Omega_\rho)$-orbit of $\Lambda_\ML$, let $\Lambda_\bullet$ be a smooth path of Lagrangian embeddings connecting $\Lambda$ and $\Lambda_\ML$. From  Corollary \ref{cor flux vanish hamiltonian}, $\Flux([\Lambda_\bullet])=0$, and from Proposition \ref{prop holonomy flux}, $\Lambda^*P_\rho$ and $\Lambda_\ML^*P_\rho$
have the same holonomy. Since ${\widetilde\sigma_{\mathrm{max}}}$ induces a parallel global section over $\Lambda_\ML$, the holonomy of $\Lambda_\ML^*P_\rho$ is trivial. Hence also the holonomy of $\Lambda^*P_\rho$ is trivial --- that is, it admits a parallel global section. Hence $\Lambda=\Lambda_{\widetilde\sigma}$ for some $\widetilde\sigma$.

Conversely, given any $\widetilde\sigma$, let $\Lambda_{\widetilde\sigma}$ be the corresponding embedding of $S$ into $(S\times S,\Omega_\rho)$. Then $\Lambda_{\widetilde\sigma}^*P_\rho$ is a flat $\R$-bundle by the observation at the beginning of this proof. In particular $\Lambda_{\widetilde\sigma}$ is Lagrangian by Corollary \ref{cor lagrangian}. Moreover, the holonomy of $\Lambda_{\widetilde\sigma}^*P_\rho$ is trivial. Since we already know (as above) that the holonomy of $\Lambda_\ML^*P_\rho$ is trivial, from Proposition \ref{prop holonomy flux}
we obtain that $\Flux([\Lambda_\bullet])=0$ for any smooth path $\Lambda_\bullet:[0,1]\to\mathcal M_\rho$ connecting $\Lambda_{\widetilde\sigma}$ and $\Lambda_\ML$. Hence from Corollary \ref{cor flux vanish hamiltonian}, $\Lambda_{\widetilde\sigma}$ and $\Lambda_\ML$ are in the same $\Ham(S\times S,\Omega_\rho)$-orbit.
\end{proof}

Then we have the following direct corollary of Theorem \ref{main thm}:

\begin{repcor}{main cor}
Let $\rho=(\rho_l,\rho_r):\pi_1(S)\to\PSL_2\R\times \PSL_2\R$, where $\rho_l$ and $\rho_r$ are Fuchsian representations. Then for every $\rho$-equivariant spacelike embedding $\sigma:\widetilde S\to\AdS^3$, $\Lambda_\sigma$ is Hamiltonian isotopic to the unique minimal Lagrangian submanifold $\Lambda_\ML$ isotopic to the diagonal.
\end{repcor}

\begin{proof}
From Lemma \ref{lemma lift embedding hor}, $\sigma$ induces a $\widetilde \rho$-equivariant embedding $\widetilde\sigma_N:\widetilde S\to \TA$, which is orthogonal to the orbits of the geodesic flow. Hence from Theorem \ref{main thm}, $\Lambda_{\widetilde\sigma_N}$ is in the same $\Ham(S\times S,\Omega_\rho)$-orbit as $\Lambda_\ML$. In other words, there exists a Hamiltonian isotopy connecting 
$\Lambda_{\widetilde\sigma_N}$ and $\Lambda_\ML$. Since $\Lambda_{\widetilde\sigma_N}=\Lambda_\sigma$, this concludes the proof.
\end{proof}

The question whether any Lagrangian submanifold $\Lambda$ of $(S\times S,\Omega_\rho)$ Hamiltonian isotopic to $\Lambda_\ML$
can be obtained as $\Lambda_\sigma$ for some $\rho$-equivariant embedding $\sigma:\widetilde S\to\AdS^3$ (hence for an embedding into $\AdS^3$, not only into $\TA$), is still open. For instance, Example \ref{ex fuchsian} shows that \emph{some} parallel section of $\Lambda^*P_\rho$ might correspond to an equivariant \emph{embedding} into $\AdS^3$, but some other may give rise to singular maps when one projects down to $\AdS^3$. (Actually, two parallel sections of $\Lambda^*P_\rho$ only differ by the action of $\R$ on the $\R$-principal bundle $P_\rho$.)

Hence the best situation one might hope is that there always exists \emph{at least one} parallel section which induces an equivariant embedding into $\AdS^3$. We do not have any positive result in this direction at the present time.

Another remark which should be made is that, if $\Lambda$ is a Lagrangian submanifold such that, for a path $\Lambda_\bullet$ connecting $\Lambda$ and $\Lambda_\ML$ one has $\Flux([\Lambda_\bullet])\in\Hom(\pi_1(S),2\pi\Z)$, then  one can take on the universal cover a parallel global section of $\TA$ over $\widetilde{\Lambda}\subset \Hyp^2\times \Hyp^2$. From Remark \ref{remark other lifts of reps}, this section is of course not $\widetilde\rho$-equivariant, but it is instead $\widetilde\rho'$-equivariant for some other lift $\widetilde\rho'$ of $\rho:\pi_1(S)\to\PSL_2\R\times \PSL_2\R$ to $\isom_0(\tAdS)$. 

Hence in this case, this global section still induces a $\rho$-equivariant map $\sigma:\widetilde S\to\AdS^3$. However, this equivariant map $\sigma$ will \emph{never} be non-singular if $\Flux([\Lambda_\bullet])\neq 0$, as a consequence of Lemma \ref{lemma lifting}.


\bibliographystyle{alpha}
\bibliography{../bs-bibliography}

\end{document}